\documentclass[12pt]{amsart}
\usepackage[OT2,T1]{fontenc}
\DeclareSymbolFont{cyrletters}{OT2}{wncyr}{m}{n}
\DeclareMathSymbol{\Sha}{\mathalpha}{cyrletters}{"58}
\usepackage[all,cmtip]{xy}
\usepackage{enumerate, comment}
\usepackage{ amssymb, latexsym, amsmath}
\usepackage{fullpage}
\usepackage{url}
\usepackage{hyperref}
\usepackage{ crimson }
\usepackage{amsfonts}
\usepackage{tikz-cd}
\newcommand{\op}[1]{\operatorname{#1}}
\usepackage{amssymb}
\usepackage{amsthm}
\usepackage{amsmath}
\usepackage{amscd}
\usepackage[mathscr]{eucal}
\usepackage{indentfirst}
\usepackage{graphicx}
\usepackage{graphics}
\usepackage{pict2e}
\usepackage{epic}
\usepackage[OT2,T1]{fontenc}
\numberwithin{equation}{section}
\usepackage[margin=3.5cm]{geometry}
\usepackage{epstopdf}

\theoremstyle{plain}
\newtheorem{Th}{Theorem}[section]
\newtheorem{Lemma}[Th]{Lemma}
\newtheorem{Cor}[Th]{Corollary}
\newtheorem{Prop}[Th]{Proposition}
\newtheorem{Remark}[Th]{Remark}
 \theoremstyle{definition}
\newtheorem{Def}[Th]{Definition}

\newtheorem{?}[Th]{Problem}

\newtheorem{Fact}[Th]{Fact}

\newcommand{\Hom}{{\rm{Hom}}}
\newcommand{\Of}{\mathcal{O}^{\flat}_{\mathbb{C}_p}}

\newcommand*{\rom}[1]{\uppercase\expandafter{\romannumeral #1\relax}}
\newcommand{\Q}{\mathbb{Q}}
\newcommand{\Z}{\mathbb{Z}}
\newcommand{\MF}{\text{MF}_{tor}^{[2-p,1]}}
\newcommand{\F}{\mathbb{F}}

\newcommand{\Fl}{\mathcal{F}^{\lambda}}
\newcommand{\C}{\mathbb{C}}
\newcommand{\B}{\text{B}_{\rm{cris}}}
\newcommand{\D}{\text{D}_{\rm{cris}}}

\newcommand{\p}{\varphi}
\renewcommand{\t}{\text{tr}\varphi}

\newcommand{\g}{\op{Ad}^0\bar{\rho}}

\newcommand{\RepQp}{\text{Rep}_{K}^f(\text{G}_{\Q_p})}
\newcommand{\RepZp}{\text{Rep}_{\op{W}(\F)}^f(\text{G}_{\Q_p})}

\DeclareSymbolFont{cyrletters}{OT2}{wncyr}{m}{n}

\DeclareMathOperator{\Gal}{Gal}

\newcommand\mtx[4] { \left( {\begin{array}{cc}
   #1 & #2 \\
   #3 & #4 \\
  \end{array} } \right)}

\begin{document}

\title{A Refined Lifting Theorem for Supersingular Galois Representations}
  
\author{Anwesh Ray}
\address{Cornell University \\ Department of Mathematics \\
  Malott Hall, Ithaca, NY 14853-4201 USA} 
\email{ar2222@cornell.edu}
\maketitle
\begin{abstract} Let $p\geq 5$ be a prime number, $\F$ a finite field of characteristic $p$ and let $\bar{\chi}$ be the mod-$p$ cyclotomic character. Let $\bar{\rho}:\op{G}_{\Q}\rightarrow \op{GL}_2(\F)$ be a Galois representation such that the local representation $\bar{\rho}_{\restriction \text{G}_{\Q_p}}$ is \textit{flat} and irreducible. Further, assume that $\text{det}\bar{\rho}=\bar{\chi}$. The celebrated theorem of Khare and Wintenberger asserts that if $\bar{\rho}$ satisfies some natural conditions, there exists a normalized Hecke-eigencuspform $f=\sum_{n\geq 1} a_n q^n$ and a prime $\mathfrak{p}|p$ in its field of Fourier coefficients such that the associated $\mathfrak{p}$-adic representation ${\rho}_{f,\mathfrak{p}}$ lifts $\bar{\rho}$. In this manuscript we prove a refined version of this theorem, namely, that one may control the valuation of the $p$-th Fourier coefficient of $f$. The main result is of interest from the perspective of the $p$-adic Langlands program.
\end{abstract}
\section{Introduction}
\subsection{Statement of the Main Result}
\par Let $p\geq 5$ be a prime number and $\F$ a finite field of characteristic $p$. Denote by $v_p$ the $p$-adic valuation on the algebraic closure $\bar{\Q}_p$ normalized by $v_p(p)=1$. Let $f$ be a normalized Hecke eigencuspform. Set $\Q(f)$ to denote the number field generated by the Fourier coefficients of $f$ and set $\mathcal{O}_f\subset \Q(f)$ to be its ring of integers. Fix a prime $\mathfrak{p}|p$ of $\Q(f)$ and denote by $\iota_{\mathfrak{p}}:\Q(f)\rightarrow \Q(f)_{\mathfrak{p}}$ the corresponding inclusion. The $\mathfrak{p}$-adic Galois representation attached to $f$ is denoted by
\[\rho_{f, \mathfrak{p}}:\text{G}_{\Q}\rightarrow \text{GL}_2(\Q(f)_{\mathfrak{p}}).\] Set $\mathcal{O}_{f,\mathfrak{p}}$ to denote the $\mathfrak{p}$-adic completion of the ring of integers $\mathcal{O}_f$.
Let $\varpi$ be a uniformizer of $\mathcal{O}_{f,\mathfrak{p}}$. The representation $\rho_{f,\mathfrak{p}}$ is continuous w.r.t. the natural topologies on $\op{G}_{\Q}$ and $\op{GL}_2(\Q(f)_{\mathfrak{p}})$. The image of $\rho_{f,\mathfrak{p}}$ is compact. It follows that there is a suitable basis w.r.t. which the image of $\rho_{f,\mathfrak{p}}$ lies in $\op{GL}_2(\mathcal{O}_{f, \mathfrak{p}})$. The mod-$\varpi$ reduction of $\rho_{f, \mathfrak{p}}:\text{G}_{\Q}\rightarrow \text{GL}_2(\mathcal{O}_{f,\mathfrak{p}})$ is denoted $\bar{\rho}_{f, \mathfrak{p}}$. The semisimplification $\bar{\rho}_{f, \mathfrak{p}}^{\op{ss}}$ is independent of the choice of basis for the underlying vector space of $\rho_{f, \mathfrak{p}}$.
\par Fontaine and Mazur observed that the continuous Galois representation $\rho_{f, \mathfrak{p}}$ satisfies a number of characteristic properties, namely:
\begin{enumerate}
\item $\rho_{f, \mathfrak{p}}$ is irreducible,
    \item $\rho_{f, \mathfrak{p}}$ is unramified away from a finite set of primes,
    \item $\rho_{f, \mathfrak{p}}$ is \textit{odd}, that is, $\op{det} \rho_{f, \mathfrak{p}}(c)=-1$, where $c$ denotes the complex-conjugation involution.
    \item The local Galois representation $\rho_{\restriction \op{G}_{\Q_p}}$ is \textit{de Rham} (see for instance \cite[p. 73]{ConradBrinon}).
    
\end{enumerate}
A continuous Galois representation satisfying the above conditions is called \textit{geometric}, see \cite{FontaineMazur}. Fontaine and Mazur conjectured that a two-dimensional geometric Galois representation arises from a Hecke eigencuspform. Let $\bar{\rho}:\text{G}_{\Q}\rightarrow \text{GL}_2(\F)$ be an odd, absolutely irreducible Galois representation. Serre conjectured that $\bar{\rho}$ lifts to a characterisic zero modular Galois representation. Ramakrishna in \cite{Ram2, RaviFM} showed that if $\bar{\rho}$ is subject to some favorable conditions, then it lifts to a geometric Galois representation. Subsequently, Khare and Wintenberger \cite{KW} proved Serre's conjecture. In particular, they proved the strong form, which asserts that the modular form in question can be arranged to have level equal to the prime to $p$ part of the Artin conductor of $\bar{\rho}$. This part of the argument follows from Ribet's level lowering theorem, see \cite{ribetLL}. Apart from a few exceptional cases, the Fontaine-Mazur conjecture has been settled, see for instance \cite{skinnerwiles, TaylorFM, KisinFM}. \par In this paper, we prove a refined version of Serre's conjecture. Set $\chi$ to denote the $p$-adic cyclotomic character and let $\bar{\chi}$ denote its mod-$p$ reduction. Throughout, $\bar{\rho}:\op{G}_{\Q}\rightarrow \op{GL}_2(\F)$ will be a representation with determinant equal to $\bar{\chi}$ and such that $\bar{\rho}_{\restriction \op{G}_{\Q_p}}$ is absolutely irreducible. Note that such Galois representations arise for instance from elliptic curves over $\Q$ with supersingular reduction at $p$. Let $E$ is an elliptic curve with good supersingular reduction at $p$, and $\rho_{E,p}$ is the associated $p$-adic Galois representation. Then the restriction $\rho_{E,p\restriction \op{G}_{\Q_p}}\mod{p^N}$ arises from a the Galois action on the generic fibre of a finite flat group-scheme over $\Z_p$ and $\rho_{E,p\restriction \op{G}_{\Q_p}}\mod{p}$ is irreducible (see \cite[Theorem 1.2]{conradflat}). Deformations that arise this way (from Galois actions coming from finite flat group schemes over $\Z_p$) are parametrized by a deformation functor called the \textit{flat deformation functor}. The representability of the flat deformation functor was worked out by Ramakrishna in \cite{Ram1}, who also provided a description for the universal deformation ring. His calculations were based on the technique of Fontaine and Laffaille \cite{FontaineLaffaille}. It turns out that flat deformations arise from certain Fontaine-Laffaille modules considered in this paper and this allows one to make explicit calculations. For a finite set of primes numbers $S$, let $\Q_S$ be the maximal algebraic extension of $\Q$ which is unramified away from $S$ and set $\op{G}_{\Q,S}:=\op{Gal}(\Q_S/\Q)$.
\begin{Th}\label{main} Let $S$ be a finite set of primes containing $p$ and $\bar{\rho}:\op{G}_{\Q,S}\rightarrow \op{GL}_2(\F)$ be a Galois representation. Assume that the following conditions are satisfied:
\begin{enumerate}
\item Assume that the image of $\bar{\rho}$ contains $\op{SL}_2(\F_p)$ and $\bar{\rho}$ is surjective if $p=5$.
\item The local representation $\bar{\rho}_{\restriction \text{G}_{\Q_p}}$ is flat and absolutely irreducible,
\item $\det \bar{\rho}=\bar{\chi}$.
\end{enumerate}
Then, for any integer $\lambda\in \Z_{\geq 1}$, there exists an eigencuspform $f=\sum_{n\geq 1} a_n q^n\in S_2(\Gamma_1(N))$ and a choice of prime $\mathfrak{p}|p$ in $\mathcal{O}_f$, such that the following are satisfied:
\begin{enumerate}
\item $N$ is coprime to $p$.
    \item The residue field $k(\mathfrak{p})$ is contained in $\F$ and the ramification index $e(\mathfrak{p}|p)$ is equal to $1$.
    \item With respect to the choice of embedding $\iota_{\mathfrak{p}}:\Q(f)\hookrightarrow \Q(f)_{\mathfrak{p}}$, the valuation $v_p(\iota_{\mathfrak{p}}(a_p))$ is equal to $\lambda$.
    \item There is a finite set of primes $X$ disjoint from $S$ such that the Galois representation $\rho_{f,\mathfrak{p}}$ is unramified at primes away from $S\cup X$.
    \item The characteristic zero representation $\rho_{f,\mathfrak{p}}$ lifts $\bar{\rho}$, as depicted
\end{enumerate}
\[\begin{tikzpicture}[node distance = 2.0cm, auto]
      \node (GSX) {$\op{G}_{\Q,S\cup X}$};
      \node (GS) [right of=GSX] {$\op{G}_{\Q,S}$};
      \node (GL2) [right of=GS]{$\op{GL}_{2}(\F).$};
      \node (GL2W) [above of= GL2]{$\op{GL}_2(\op{W}(\F))$};
      \draw[->] (GSX) to node {} (GS);
      \draw[->] (GS) to node {$\bar{\rho}$} (GL2);
      \draw[->] (GL2W) to node {} (GL2);
      \draw[dashed,->] (GSX) to node {$\rho_{f, \mathfrak{p}}$} (GL2W);
      \end{tikzpicture}\]
\end{Th}

\par There has been some interest in computing the local reductions of modular Galois representations attached to an eigencuspform $f$ for which the slope of the $p$-th Fourier coefficient is prescribed. Such questions are of interest from the perspective of the $p$-adic Langlands program. The following is a consequence of the main theorem of \cite{berger2004construction}.
\begin{Th}(Berger-Li-Zhu)
 Let $f$ and $g$ are normalized eigencuspforms of weight $k\geq 2$ and $\mathfrak{p}|p$ be a prime in $\Q(f)\cdot \Q(g)$ such that the local representations $\rho_{f,\mathfrak{p}\restriction \op{G}_{\Q_p}}$ and $\rho_{g,\mathfrak{p}\restriction \op{G}_{\Q_p}}$ are crystalline and residually irreducible. Set $\lambda_f:=v_p(\iota_{\mathfrak{p}}(a_p(f)))$ (resp. $\lambda_g:=v_p(\iota_{\mathfrak{p}}(a_p(g)))$), where $a_p(f)$ (resp. $a_p(g)$) is the $p$-th Fourier coefficient of $f$ (resp. $g$).
 Suppose that $\lambda_f,\lambda_g> \lfloor (k-2)/(p-1)\rfloor$. Then, the residual representations $\bar{\rho}_{f,\mathfrak{p}\restriction \op{G}_{\Q_p}}$ and $\bar{\rho}_{g,\mathfrak{p}\restriction \op{G}_{\Q_p}}$ are isomorphic up to a twist by a character.
\end{Th}Thus in particular, when $\lfloor (k-2)/(p-1)\rfloor=0$, i.e., $k\leq p$, the residual representation $\bar{\rho}_{f,\mathfrak{p}\restriction \op{G}_{\Q_p}}$ is uniquely determined up to twist by a character. A more explicit description is also given for these reductions, see \cite[pp.1-2]{berger2004construction}. The results in this paper are in the opposite direction. Theorem $\ref{main}$ shows that all integral slopes are realized by eigencuspforms when $k=2$ when the residual representation (and not only the local reduction at $p$) is fixed. There has been much interest in refining the results of Berger-Li-Zhu, the study of local reductions of modular Galois representations has gained considerable momentum in \cite{BuzzardGee, BuzzardGee2,GanguliGhate,BG15, BGR, Ros}.
\par We present a comprehensive overview of the strategy of proof in section $\ref{sectionoverview}$, where we also explain how the paper is organized.

\subsection{Acknowledgements}
I am very grateful to my advisor Ravi Ramakrishna for introducing me to the fascinating subject of Galois deformations and for some helpful suggestions. I would also like to thank Brian Hwang, Aftab Pande and Stefan Patrikis for some fruitful conversations. Finally, I would like to thank the anonymous referees for their insightful suggestions which have led to considerable improvement of this manuscript.
\section{Notation}
\par In this section we summarize some basic notation in this manuscript.
\begin{itemize}
    \item The $p$-adic valuation $v_p:\bar{\Q}_p\rightarrow \Q\cup\{\infty\}$ is normalized by $v_p(p)=1$.
    \item The completion of $\bar{\Q}_p$ w.r.t the valuation $v_p$ is denoted by $\C_p$. We simply denote the extension of the valuation to $\C_p$ by $v_p$ which takes values in $\mathbb{R}\cup \{\infty\}$. The valuation ring is denoted by $\mathcal{O}_{\C_p}$.
    \item The ring of Witt vectors with residue field $\F$ is denoted by $\op{W}(\F)$ and $K:=\op{W}(\F)[1/p]$. The field $K$ is the unramified extension of $\Q_p$ with residue field $\F$.
   \item Let $\F[\epsilon]/(\epsilon^2)$ denote the ring of dual numbers over $\F$. We shall simply denote the ring by $\F[\epsilon]$ with the understanding that $\epsilon^2=0$.

\item Let $\Sigma$ be a finite set of primes containing $S$ and $M$ an $\F[\text{G}_{\Q,\Sigma}]$-module, for $i=0,1,2$, the $\Sha^i$-group is the kernel of the restriction map
\[\Sha_{\Sigma}^i(M):=\text{ker}\{H^1(\op{G}_{\Q,\Sigma},M)\rightarrow \bigoplus_{v\in \Sigma} H^1(\op{G}_{\Q_v},M)\}.\]
\end{itemize}
\section{Overview}\label{sectionoverview}
\par Let $\bar{\rho}$ be a Galois representation $\bar{\rho}:\operatorname{G}_{\Q}\rightarrow \op{GL}_2(\F)$ which satisfies the conditions of Theorem \ref{main}. It follows from the well known results in \cite{Ram2, RaviFM} that $\bar{\rho}$ lifts to a continuous Galois representation $\rho:\op{G}_{\Q}\rightarrow \op{GL}_2(\op{W}(\F))$ with $\det\rho=\chi$ which is geometric in the sense of Fontaine and Mazur. This means that $\rho$ is odd, unramified away from finitely many primes and the local representation $\rho_{\restriction \operatorname{G}_{\Q_p}}$ is de Rham. We remark that at the time Ramakrishna proved his results on lifting Galois representations, the Fontaine-Mazur conjecture had not yet been settled. One step in the construction is to analyze suitable local deformation conditions associated to the local representations $\bar{\rho}_{\restriction \op{G}_{\Q_l}}$ at each prime $l$ at which $\bar{\rho}$ is ramified. Local at $p$ conditions are comprehensively studied in \cite[pp. 124-138]{RaviFM}. In \cite{Ram1}, he analyzed the \textit{flat deformation condition} at $p$, cf. Definition $\ref{flatdef}$. Since $\bar{\rho}_{\restriction \op{G}_{\Q_p}}$ is assumed to be flat, it follows that $\rho_{\restriction \op{G}_{\Q_p}}$ can be arranged to be flat, and in particular, crystalline. The analysis of flat deformations via the theory of Fontaine-Laffaille modules aided in the calculations made in \cite{Ram1}, and this method is revisited in sections $\ref{section3}$ and $\ref{section4}$ of this paper. Since $\bar{\rho}_{\restriction \op{G}_{\Q_p}}$ is assumed to be irreducible, the Fontaine Mazur conjecture is known for geometric lifts of $\bar{\rho}$ and it follows from the main theorem in \cite{KisinFM} that $\rho$ does arise from a Hecke eigencuspform $f$ of weight $2$ on $\Gamma_1(N)$. Furthermore, $f$ is supersingular at $p$ and $N$ is prime to $p$. In this section we explain the general lifting strategy used in proving Theorem $\ref{main}$, after reviewing a few preliminary notions.
\subsection{Preliminaries}
\begin{Def} Let $\mathcal{C}_{\op{W}(\F)}$ be the category of coefficient rings over $\op{W}(\F)$ with residue field $\F$. The objects of this category consist of local $\text{W}(\F)$-algebras $(R,\mathfrak{m})$ such that
      \begin{itemize}
          \item $R$ is complete and Noetherian,
          \item $R/\mathfrak{m}$ is isomorphic to $\F$ as a $\text{W}(\F)$-algebra. The residual map \[\phi:R\rightarrow \F\]is the composite of the quotient map $R\rightarrow R/\mathfrak{m}$ with the unique isomorphism of $\op{W}(\F)$-algebras $R/\mathfrak{m}\xrightarrow{\sim}\F$.
      \end{itemize} 
      A finite length coefficient ring is given the discrete topology. An arbitrary coefficient ring is an inverse limit of finite length coefficient rings and is given the inverse limit topology. A morphism $F:(R_1,\mathfrak{m}_1)\rightarrow (R_2,\mathfrak{m}_2)$ is a continuous homorphism of local rings which is also a $\text{W}(\F)$-algebra homorphism. The subcategory of finite length coefficient rings is denoted $\mathcal{C}_{\op{W}(\F)}^f$.      \end{Def} 
      \par Recall that $\chi:\op{G}_{\Q}\rightarrow \op{GL}_1(\op{W}(\F))$ is the $p$-adic cyclotomic character. For $R\in \mathcal{C}_{\op{W}(\F)}$, let $\chi_R$ denote the character obtained by composing $\chi$ with the natural map $\op{GL}_1(\op{W}(\F))\rightarrow \op{GL}_1( R)$ induced by the structure map $\op{W}(\F)\rightarrow R$. Let $v$ be a prime number and denote by $\chi_{R,v}$ the restriction of $\chi_R$ to $\op{G}_{\Q_v}$. Denote by $\phi^*: \op{GL}_2(R)\rightarrow \op{GL}_2(\F)$ the group homomorphism induced by the residual homomorphism $\phi: R\rightarrow \F$. We say that a continuous homomorphism $\rho_R:\op{G}_{\Q_v}\rightarrow \op{GL}_2(R)$ is an $R$-lift of $\bar{\rho}_{\restriction \operatorname{G}_{\Q_v}}$ if $\phi^*\circ \rho_R=\bar{\rho}_{\restriction \operatorname{G}_{\Q_v}}$, i.e., the following diagram commutes
     \[ \begin{tikzpicture}[node distance = 2.2 cm, auto]
            \node(G) at (0,0){$\op{G}_{\Q_v}$};
             \node (A) at (3,0) {$\op{GL}_2(\F)$.};
             \node (B) at (3,2) {$\op{GL}_2(R)$};
      \draw[->] (G) to node [swap]{$\bar{\rho}_{\restriction \operatorname{G}_{\Q_v}}$} (A);
       \draw[->] (B) to node{$\phi^*$} (A);
      \draw[->] (G) to node {$\rho_R$} (B);
      \end{tikzpicture}\]Further, we require that $\det\rho_R$ is equal to $\chi_{R,v}$. Two lifts $\rho_R$ and $\rho_R'$ are said to be \textit{strictly-equivalent} if there is
      \[A\in \widehat{\op{GL}_2}(R):=\op{ker}\lbrace \op{GL}_2(R)\xrightarrow{\phi^*} \op{GL}_2(\F)\rbrace \]
      such that 
      $\rho_R=A\rho_R' A^{-1}$. A \textit{deformation} is a strict equivalence class of lifts. Let $\operatorname{Def}_v(R)$ be the set of $R$-deformations of $\bar{\rho}_{\restriction \op{G}_{\Q_v}}$ with determinant $\chi_{R,v}$. The association $R\mapsto \operatorname{Def}_v(R)$ defines a covariant functor \[\operatorname{Def}_v:\mathcal{C}_{\op{W}(\F)}\rightarrow \operatorname{Sets}. \]
      The adjoint representation $\text{Ad}^0\bar{\rho}$ is the $\F[\text{G}_{\Q}]$-module of trace zero matrices
    \[\op{Ad}^0\bar{\rho}=\left\{ \mtx{a}{b}{c}{-a}\mid a,b,c\in \F\right\}\]where $g\in \text{G}_{\Q}$ acts through conjugation via $\bar{\rho}$
    \[g\cdot X=\bar{\rho}(g)X\bar{\rho}(g)^{-1}.\]The tangent space $\operatorname{Def}_v(\F[\epsilon]/(\epsilon^2))$ naturally acquires the structure of an $\F$-vector space and is isomorphic to $H^1(\operatorname{G}_{\Q_v},\g)$. Under this association, a cohomology class $f$ is identified with the deformation $(\operatorname{Id}+\epsilon f)\bar{\rho}_{\restriction \op{G}_{\Q_v}}$. For $m\in \Z_{\geq 2}$, the set of deformations $\operatorname{Def}_v(\text{W}(\F)/p^{m})$ is equipped with action of the cohomology group $H^1(\op{G}_{\Q_v}, \g)$. For $\varrho_m\in \operatorname{Def}_v(\text{W}(\F)/p^{m})$ and $f\in H^1(\op{G}_{\Q_v}, \g)$, the twist of $\varrho_m$ by $f$ is defined by the formula $(\op{Id}+p^{m-1}f)\varrho_m$. The set of deformations $\varrho_m$ of a fixed $\varrho_{m-1}\in \operatorname{Def}_v(\text{W}(\F)/p^{m-1})$ is either empty or in bijection with $H^1(\operatorname{G}_{\Q_v},\g)$.
      \begin{Def}\label{defconditiondef} (see \cite{taylor}) A sub-functor $\mathcal{C}_v$ of $\operatorname{Def}_v$ is referred to as a \textit{deformation functor}. It is a \textit{deformation condition} if (1) to (3) below are satisfied. If condition (4) is satisfied, $\mathcal{C}_v$ is said to be a \textit{liftable deformation functor}.
      \begin{enumerate}
          \item First, we require that $\mathcal{C}_v(\F)=\{\bar{\rho}_{\restriction \op{G}_{\Q_v}}\}.$
          \item Let $R_1, R_2\in\mathcal{C}_{\op{W}(\F)}$, let $\rho_1\in \mathcal{C}_v(R_1)$ and $\rho_2\in \mathcal{C}_v(R_2)$. Let $I_1$ be an ideal in $R_1$ and $I_2$ an ideal in $R_2$ such that there is an isomorphism $\alpha:R_1/I_1\xrightarrow{\sim} R_2/I_2$ satisfying \[\alpha(\rho_1 \;\text{mod}{I_1})=\rho_2 \;\text{mod}{I_2}.\] Let $R_3$ be the fibred product \[R_3=\lbrace(r_1,r_2)\mid \alpha(r_1\text{mod} I_1)=r_2 \text{mod} I_2\rbrace\] and $\rho_3$ the $R_3$-deformation induced from $\rho_1$ and $\rho_2$. Then $\rho_3$ satisfies $\mathcal{C}_v(R_3)$.
          \item Let $R\in \mathcal{C}_{\op{W}(\F)}$ with maximal ideal $\mathfrak{m}_R$. If $\rho\in \operatorname{Def}_v(R)$ is such that $\rho\mod{\mathfrak{m}_R^n}$ satisfies $\mathcal{C}_v$ for all $n\in \Z_{\geq 1}$, then  $\rho$ also satisfies $\mathcal{C}_v$.
          \item\label{def22c4} Let $R\in \mathcal{C}_{\op{W}(\F)}$ and $I$ an ideal such that $I.\mathfrak{m}_R=0$. For $\rho\in \mathcal{C}_v(R/I)$, there exists $\tilde{\rho}\in \mathcal{C}_v(R)$ such that $\rho=\tilde{\rho}\mod{I}$.
      \end{enumerate}
      \end{Def}
      Let $\mathcal{C}_v$ be a local deformation condition at the prime $v$. The tangent space $\mathcal{N}_v$ consists of $f\in H^1(\op{G}_{\Q_v}, \g)$, such that $(\op{Id}+\epsilon f) \bar{\rho}_{\restriction \op{G}_{\Q_v}}\in \mathcal{C}_v(\F[\epsilon]/(\epsilon^2))$.
      \par 
      Let $m>1$ be an integer. In the discussion that follows, $m$ is fixed. Identify $\g$ with the kernel of the mod-$p^{m-1}$ reduction map
      \[\op{SL}_2(\op{W}(\F)/p^m)\rightarrow\op{SL}_2(\op{W}(\F)/p^{m-1}) \]by identifying $\mtx{a}{b}{c}{-a}$ in $\g$ with  \[\op{Id}+p^{m-1}\mtx{a}{b}{c}{-a}=\mtx{1+p^{m-1}a}{p^{m-1}b}{p^{m-1}c}{1-p^{m-1}a}.\] The action of $\mathcal{N}_v$ on $\operatorname{Def}_v(\text{W}(\F)/p^m)$ is induced from that of $H^1(\op{G}_{\Q_v}, \g)$. This action stabilizes $\mathcal{C}_v(\text{W}(\F)/p^m)$, i.e., if $\varrho_m\in \mathcal{C}_v(\text{W}(\F)/p^m)$ and $f\in \mathcal{N}_v$, then 
      \[(\op{Id}+p^{m-1}f)\varrho_{m}\in \mathcal{C}_v(\text{W}(\F)/p^m). \]
      Let $S$ be a finite set of primes containing $p$ and the primes at which $\bar{\rho}$ is ramified. The global deformation theoretic method of Ramakrishna requires the existence of suitable local deformation functors $\mathcal{C}_v$ at each prime $v\in S$. The following summarizes the precise condition required of the local deformation functors.
      \begin{Prop}\label{prop23}
      Let $v$ be a prime in $S\backslash \{p\}$. There is a liftable local deformation condition $\mathcal{C}_v\subseteq \op{Def}_v$ such that \[\dim  \mathcal{N}_v=\op{dim}H^0(\op{G}_{\Q_v}, \g).\] When $v=p$, there is a deformation functor $\mathcal{C}_p\subseteq \op{Def}_p$ consisting of flat deformations. This is a liftable deformation condition such that \[\dim  \mathcal{N}_p=\op{dim}H^0(\op{G}_{\Q_p}, \g)+1.\]
      \end{Prop}
      \begin{proof}
      For the case $v\neq p$, see \cite[Proposition 1, Remark p.124]{RaviFM}. For $v=p$, the functor $\mathcal{C}_p$ is the flat deformation functor with fixed determinant considered in \cite{Ram1}.
      \end{proof} 
      From here on in, for $v\in S\backslash \{p\}$, the functor $\mathcal{C}_v$ will denote the deformation condition from the above Proposition. We de not claim that the condition $\mathcal{C}_v$ is uniquely determined, but we simply work with specific choice of $\mathcal{C}_v$ as defined in \cite{RaviFM}. At $p$, the functor $\mathcal{C}_p$ will denote the flat deformation condition from \cite{Ram1}. We shall work with certain subfunctors of $\mathcal{C}_p$, yet to be defined.
\subsection{The Lifting Strategy}\label{section32}
\par Fix an integer $\lambda\in \Z_{\geq 1}$. Ramakrishna's methods provide us with a geometric lift which is flat at $p$. Theorem $\ref{main}$ is proved via an adaptation of Ramakrishna's lifting strategy, which we now explain. It is shown that there exists a suitable subfunctor $\mathcal{C}_{p}^{\lambda}$ of the functor of flat deformations $\mathcal{C}_p$ with the following properties:
\begin{enumerate}
    \item\label{cplambda1} Let $f=\sum_{n\geq 1} a_n q^n $ be a normalized Hecke eigencuspform and $\mathfrak{p}|p$ a prime of $\Q(f)$. Assume that the $\mathfrak{p}$-adic Galois representation $\rho_{f, \mathfrak{p}}$ lifts $\bar{\rho}$. Proposition $\ref{comparisonprop}$ shows that if $\rho_{f, \mathfrak{p}}$ satisfies $\mathcal{C}_p^{\lambda}$ when retricted to $\op{G}_{\Q_p}$, then the slope of $f$ is equal to $\lambda$. In other words, we have that \[v_p(\iota_{\mathfrak{p}}(a_p))=\lambda\] if $\rho_{f, \mathfrak{p}}$ satisfies $\mathcal{C}_p^{\lambda}$.
    \item\label{cplambda2} The deformation functor $\mathcal{C}_p^{\lambda}$ satisfies condition $\eqref{def22c4}$ of Definition $\ref{defconditiondef}$, i.e., it is a liftable deformation functor. Proposition $\ref{liftableprop}$ shows that this condition is satisfied. Note that we do not claim that it is a deformation condition, in the sense of Definition $\ref{defconditiondef}$.
    \item\label{cplambda3} Let $\mathcal{N}_p$ be the tangent space to the flat deformation functor $\mathcal{C}_p$. We require that $\mathcal{N}_p$ stabilizes mod $p^m$ lifts of $\bar{\rho}_{\restriction \op{G}_{\Q_p}}$ in $\mathcal{C}_p^{\lambda}$ for $m\geq  \lambda+2$. In greater detail, this condition requires that if $m\geq  \lambda+2$, $X\in \mathcal{N}_p$ and $\varrho_m\in \mathcal{C}_p^{\lambda}(\op{W}(\F)/p^m)$, then the twist $(\text{Id}+p^{m-1}X)\varrho_m$ also satisfies $\mathcal{C}_p^{\lambda}(\op{W}(\F)/p^m)$. Proposition $\ref{localatp}$ shows this property is satisfied.
\end{enumerate}Recall that for $v\in S\backslash \{p\}$, the functor $\mathcal{C}_v$ is specified as in Proposition $\ref{prop23}$. The main theorem is proved by lifting $\bar{\rho}$ to a geometric representation $\rho$ by successively lifting $\rho_m$ to $\rho_{m+1}$ as depicted
\[\begin{tikzpicture}[node distance = 2.0 cm, auto]
      \node (GSX) at (0,0){$\operatorname{G}_{\Q,S\cup X}$};
      \node (GL2) at (5,0){$\op{GL}_2(\F),$};
      \node (GL2Wn) at (3,2)[above of= GL2]{$\op{GL}_2(\op{W}(\F)/p^{m})$};
      \node (GL2Wnplus1) at (5,4){$\op{GL}_2(\op{W}(\F)/p^{m+1})$};
      \draw[->] (GSX) to node [swap]{$\bar{\rho}$} (GL2);
      \draw[->] (GL2Wn) to node {} (GL2);
      \draw[->] (GSX) to node [swap]{$\rho_m$} (GL2Wn);
      \draw[->] (GL2Wnplus1) to node {} (GL2Wn);
      \draw[dashed,->] (GSX) to node {$\rho_{m+1}$} (GL2Wnplus1);
      \end{tikzpicture}\] 
      so that $\rho_{m\restriction \op{G}_{\Q_p}}$ satisfies $\mathcal{C}_p^{\lambda}$, the restriction $\rho_{m\restriction \op{G}_{\Q_v}}$ satisfies $\mathcal{C}_v$ at each prime $v\in S\backslash \{p\}$ and the auxiliary set of primes $X$ is finite. Furthermore, each auxiliary prime $v\in X$ is equipped with a liftable subfunctor $\mathcal{C}_v$ of $\operatorname{Def}_v$ whose tangent space has dimension $\dim \mathcal{N}_v=\op{dim}H^0(\op{G}_{\Q_v}, \g)$. These auxiliary primes are dubbed \textit{nice primes} by Ramakrishna and the deformations at these primes are dubbed \textit{nice deformations}. They were introduced in the two-dimensional setting in \cite[section 3]{Ram2}. We recall the definition.
      \begin{Def}\label{nicedefs}
Let $v\notin S$ be a prime and $\sigma_v$ denote the Frobenius at $v$. Then $v$ is a nice prime if:
\begin{enumerate}
    \item $v\not\equiv \pm 1\mod{p}$,
    \item for a suitable choice of basis, $\bar{\rho}(\sigma_v)$ is given by the matrix $\mtx{vx}{}{}{x}$ such that $x^2=1$.
\end{enumerate}
Let $\op{I}_v$ be the inertia group at $v$. The deformations $\mathcal{C}_v$ consist of deformations $\varrho:\op{G}_{\Q_v}\rightarrow \op{GL}_2(R)$ such that 
\[\varrho(\sigma_v)=\mtx{v\tilde{x}}{}{}{\tilde{x}}\text{ and }\varrho(g)=\mtx{1}{\ast}{}{1},\]where $\tilde{x}$ lifts $x$. Since it is assumed that $\det \varrho=\chi_{R,v}$, it follows that $\tilde{x}^2=1$. Let $\mathcal{N}_v$ denote the tangent space of $\mathcal{C}_v$.
      \end{Def}
      Let $v$ be a nice prime and denote by $\bar{\chi}_v$ the restriction of $\bar{\chi}$ to $\op{G}_{\Q_v}$. As an $\op{G}_{\Q_v}$-module, $\g$ decomposes into a direct sum $\g\simeq \F(\bar{\chi}_v)\oplus \F \oplus \F(\bar{\chi}_v^{-1})$. Standard arguments involving the use of the local Euler characteristic formula and local Tate-duality (see the proof of \cite[Lemma 2]{KLR}) show that $H^1(\op{G}_{\Q_v}, \F(\bar{\chi}_v^{-1}))=0$. Note that the classes in $H^1(\op{G}_{\Q_v}, \F)$ are unramified. For $h\in H^1(\op{G}_{\Q_v}, \g)$, we shall take $h(\sigma_v)$ to mean the value at $\sigma_v$ of the projection of $h$ to the first factor of \[H^1(\op{G}_{\Q_v}, \F)\oplus H^1(\op{G}_{\Q_v}, \F(\bar{\chi}_v)).\]
      \begin{Lemma}\label{Nvnice}
      Let $v$ be a nice prime. The space $\mathcal{N}_v$ consists of cohomology classes $h\in H^1(\op{G}_{\Q_v}, \g)$ such that $h(\sigma_v)=0$.
      \end{Lemma}
      \begin{proof}
           Recall that $h\in \mathcal{N}_v$ if and only if $(\op{Id}+\epsilon h)\bar{\rho}\in \mathcal{C}_v$. This forces that the diagonal summand of $h$ be zero and thus $h(\sigma_v)=0$. Conversely, if $h(\sigma_v)=0$ then the projection to the first factor of $h$ to \[H^1(\op{G}_{\Q_v}, \F)\oplus H^1(\op{G}_{\Q_v}, \F(\bar{\chi}_v))\] is zero. Then it is clear that $(\op{Id}+\epsilon h)\bar{\rho}\in \mathcal{C}_v$.
      \end{proof}
\par Let $X$ be a finite set of nice primes disjoint from $S$. We shall consider deformations satisfying the local conditions $\mathcal{C}_v$ for $v\in (S\backslash\{p\})\cup X$ and $\mathcal{C}_p^{\lambda}$ at $p$. Here, for $v\in S\backslash \{p\}$, the functor $\mathcal{C}_v$ is specified as in Proposition $\ref{prop23}$ and for $v\in X$, the functor $\mathcal{C}_v$ consists of nice deformations as in Definition $\ref{nicedefs}$ for $v\in X$ and prescribed by Proposition \ref{prop23} for $v\in S\backslash\{p\}$. For $v\in (S\backslash\{p\})\cup X$, set $\mathcal{N}_v\subseteq H^1(\op{G}_{\Q_v}, \g)$ for the tangent space of $\mathcal{C}_v$. Set $\mathcal{N}_p$ to be the tangent space of the functor of flat deformations $\mathcal{C}_p$. For $v\in S\cup X$, set $\mathcal{N}_v^{\perp}\subseteq H^1(\op{G}_{\Q_v},\g^*)$ to be the orthogonal complement of $\mathcal{N}_v$ with respect to the non-degenerate Tate pairing 
\[H^1(\op{G}_{\Q_v}, \g)\times H^1(\op{G}_{\Q_v}, \g^*)\rightarrow H^2(\op{G}_{\Q_v}, \F(\bar{\chi}))\xrightarrow{\sim}\F.\] Set $\mathcal{N}_{\infty}=0$ and $\mathcal{N}_{\infty}^{\perp}=0$. The Selmer-condition $\mathcal{N}$ is the tuple $\{\mathcal{N}_v\}_{v\in S\cup X\cup \{\infty\}}$ and the dual Selmer condition $\mathcal{N}^{\perp}$ is $\{\mathcal{N}_v^{\perp}\}_{v\in S\cup X\cup\{\infty\}}$. Attached to $\mathcal{N}$ and $\mathcal{N}^{\perp}$ are the Selmer and dual-Selmer groups defined as follows:
      \[H^1_{\mathcal{N}}(\op{G}_{\Q,S\cup X}, \g):=\text{ker}\left\{ H^1(\operatorname{G}_{\Q, S\cup X}, \g)\xrightarrow{\operatorname{res}_{S\cup X}} \bigoplus_{v\in S\cup X} \frac{H^1(\op{G}_{\Q_v}, \g)}{\mathcal{N}_v}\right\}\]
      and
      \[H^1_{\mathcal{N}^{\perp}}(\op{G}_{\Q,S\cup X}, \g^*):=\text{ker}\left\{ H^1(\operatorname{G}_{\Q, S\cup X}, \g^{*})\xrightarrow{\op{res}_{S\cup X}'} \bigoplus_{v\in S\cup X} \frac{H^1(\op{G}_{\Q_v}, \g^*)}{\mathcal{N}_v^{\perp}}\right\}\]
      respectively. 
      \begin{Lemma}
      The dimensions of the Selmer group and dual Selmer group coincide, i.e.,
      \begin{equation}\label{selmerisdselmer}\op{dim}H^1_{\mathcal{N}}(\op{G}_{\Q,S\cup X}, \g)=\op{dim}H^1_{\mathcal{N}^{\perp}}(\operatorname{G}_{\Q, S\cup X}, \g^{*}).\end{equation}
      \end{Lemma}
      
      \begin{proof}The following formula is due to Wiles (see \cite[Theorem 8.7.9]{NW}):
      \begin{equation}\label{wilesformula}\begin{split}&\op{dim}H^1_{\mathcal{N}}(\op{G}_{\Q,S\cup X}, \g)-\op{dim}H^1_{\mathcal{N}^{\perp}}(\operatorname{G}_{\Q, S\cup X}, \g^{*})\\&=\op{dim}H^0(\op{G}_{\Q}, \g)-\op{dim}H^0(\op{G}_{\Q},\g^*)\\ &+\sum_{v\in S\cup X\cup \{\infty\}} \left(\dim \mathcal{N}_v-\op{dim}H^0(\op{G}_{\Q_v}, \g)\right).\\\end{split}\end{equation} Since $\det \bar{\rho}=\bar{\chi}$ is odd, one has that $\op{dim}H^0(\op{G}_{\infty}, \g)=1$. It follows from Proposition $\ref{prop23}$ that 
      \[\dim \mathcal{N}_v-\op{dim}H^0(\op{G}_{\Q_v}, \g)=\begin{cases} 0 & \text{ if }v\neq p, \infty,\\
      1 & \text{ if }v=p,\\
      -1 & \text{ if }v= \infty.\\
      \end{cases}\]
      Recall that the image of $\bar{\rho}$ contains $\op{SL}_2(\F_p)$. It is an easy exercise to check that \[H^0(\op{G}_{\Q}, \g)=0\text{, and }H^0(\op{G}_{\Q},\g^*)=0.\] Putting it all together, the result follows.
      \end{proof}
      The Selmer and dual Selmer groups fit into a five-term exact sequence called the Poitou-Tate sequence (see the proof of \cite[Lemma 1.1]{taylor}). We need only point out that the cokernel of $\op{res}_{S\cup X}$ injects into $H^1_{\mathcal{N}^{\perp}}(\op{G}_{\Q,S\cup X}, \g^*)^{\vee}$. In particular, if the Selmer group is zero, then so is the dual Selmer group, in which case the restriction map $\operatorname{res}_{S\cup X}$ is an isomorphism. All deformations $\rho_m$ discussed in this paper will have similitude character equal to $\chi\mod{p^m}$.
      \par The three main steps are as follows:
      \begin{enumerate}
          \item Corollary $\ref{liftrholambda}$ shows that there is a finite set of nice primes $Z$ disjoint from $S$ and a mod $p^{\lambda+2}$ lift $\rho_{\lambda+2}:\op{G}_{\Q,S\cup Z}\rightarrow \op{GL}_2(\op{W}(\F)/p^{\lambda+2})$ of $\bar{\rho}$. This lift $\rho_{\lambda+2}$ is furthermore arranged to satisfy the conditions $\mathcal{C}_v$ at the primes $v\in (S\backslash \{p\})\cup Z$ and $\mathcal{C}_p^{\lambda}$ at $p$.
          This strategy for producing such a lift is based on the method developed by Khare, Larsen and Ramakrishna in \cite{KLR}.
          \item We show that there is a finite set of nice primes $X$ containing $Z$ such that the Selmer group $H^1_{\mathcal{N}}(\op{G}_{\Q,S\cup X}, \g)$ is zero. It follows from \eqref{selmerisdselmer} that $H^1_{\mathcal{N}^{\perp}}(\op{G}_{\Q,S\cup X}, \g^*)$ is zero. This is a standard argument, see the proofs of \cite[Lemma 1.2]{taylor} or \cite[Proposition 5.2, Lemma 5.3]{PatEx}.
          \item The space $\mathcal{N}_p$ stabilizes the mod-$p^{\lambda+2}$ lifts in $\mathcal{C}_p^{\lambda}$. The method of Ramakrishna produces a lift $\rho:\op{G}_{\Q,S\cup X}\rightarrow \op{GL}_2(\op{W}(\F))$ which satisfies the condition $\mathcal{C}_p^{\lambda}$. The Fontaine-Mazur conjecture predicts that $\rho$ arises from a normalized Hecke eigencuspform $f=\sum_{n\geq 1} a_n q^n$. This has been proved by Kisin, cf. \cite{KisinFM}. Since the condition $\mathcal{C}_p^{\lambda}$ is satisfied at $p$, it follows that $v_p(\iota_{\mathfrak{p}}(a_p))$ is equal to $\lambda$ (cf. condition $\eqref{cplambda1}$ on p.6).
      \end{enumerate}
      \par Lifting $\rho_{\lambda+2}$ to characteristic zero involves a standard argument which goes back to the original works of Ramakrishna. We review this construction in complete detail. Since $X$ is chosen so that the dual Selmer group $H^1_{\mathcal{N}^{\perp}}(\op{G}_{\Q,S\cup X}, \g^*)$ is zero, it follows from the Poitou Tate long exact sequence that the restriction map 
      \[H^1(\operatorname{G}_{\Q, S\cup X}, \g)\xrightarrow{\operatorname{res}_{S\cup X}} \bigoplus_{v\in S\cup X} \frac{H^1(\op{G}_{\Q_v}, \g)}{\mathcal{N}_v}\]is surjective. Let $m\geq \lambda+2$ and $\rho_m:\op{G}_{\Q,S\cup X}\rightarrow \op{GL}_2(\op{W}(\F)/p^m)$ be a lift of $\rho_{\lambda+2}$ which is unramified outside $S\cup X$ and satisfies the conditions $\mathcal{C}_v$ at each prime $v\in (S\backslash \{p\})\cup X$ and the condition $\mathcal{C}_p^{\lambda}$ at $p$.
      \par We show that $\rho_m$ may be lifted to $\rho_{m+1}$ which satisfies the same conditions. One may always choose a continuous lift $\tau$ of $\rho_m$ as depicted
\[ \begin{tikzpicture}[node distance = 2.6 cm, auto]
            \node(G) at (0,0) {$\operatorname{G}_{\Q,S\cup X}$};
             \node (A) at (3,0) {$\op{GL}_{2}(\op{W}(\F)/p^{m})$};
             \node (B) at (3,2){$\op{GL}_2(\op{W}(\F)/p^{m+1})$};
      \draw[->] (G) to node [swap]{$\rho_m$} (A);
       \draw[->] (B) to node{} (A);
      \draw[->] (G) to node {$\tau$} (B);
      \end{tikzpicture}\]such that the composite $\det\tau=\chi \mod{p^{m+1}}$. Note that the continuous lift $\tau$ always exists since we do not insist that it is a homomorphism. We define a cohomological obstruction to there existing a lift $\rho_{m+1}=\tau$ which is a homomorphism such that $\det\rho_{m+1}=\chi \mod{p^{m+1}}$. Identify $\g$ with the kernel of the mod-$p^m$ reduction map \[\op{SL}_2(\op{W}(\F)/p^{m+1})\rightarrow \op{SL}_2(\op{W}(\F)/p^{m}),\]so that $X\in \g$ is identified with $\op{Id}+p^m X$.
      The obstruction class \[\mathcal{O}(\rho_{m})_{\restriction S\cup X_1}\in H^2(\op{G}_{S\cup X_1}, \g)\] is represented by the $2$-cocycle
      \[(g_1,g_2)\mapsto \tau(g_1 g_2)\tau(g_2)^{-1}\tau(g_1)^{-1}.\] Note that this class is well defined and does not depend on the choice of $\tau$. The obstruction class is zero precisely when $\rho_m$ lifts one more step (so that it is still unramified away from $S\cup X$). Notice that since $\rho_m$ satisfies liftable deformation conditions at each prime $v\in S\cup X$, it follows that $\mathcal{O}(\rho_m)\in \Sha^2_{S\cup X} (\g)$. Since the dual Selmer group $H^1_{\mathcal{N}^{\perp}}(\op{G}_{\Q,S\cup X}, \g^*)$ is zero, so is $\Sha^1_{S\cup X}(\g^*)$, and it follows from global-duality of $\Sha$-groups that $\Sha^2_{S\cup X}(\g)$ is zero. Hence $\rho_m$ lifts to $\rho_{m+1}:\op{G}_{\Q,S\cup X}\rightarrow \op{GL}_2(\op{W}(\F)/p^{m+1})$.
      \par In order to complete the inductive step, it suffices to show that a lift $\rho_{m+1}$ exists so that it satisfies $\mathcal{C}_v$ for $v\in (S\backslash \{p\})\cup X$ and $\mathcal{C}_p^{\lambda}$ at $p$. After picking a suitable global cohomology class $z\in H^1(\operatorname{G}_{\Q, S\cup X}, \g)$ and replacing $\rho_{m+1}$ by its twist $(\operatorname{Id}+p^{m}z)\rho_{m+1}$, this may be arranged, as we proceed to explain. At each prime $v\in (S\backslash\{p\})\cup X$, there is a cohomology class $z_v\in H^1(\op{G}_{\Q_v}, \g)$ such that the twist $(\operatorname{Id}+p^mz_v){\rho_{m+1}}_{\restriction \op{G}_{\Q_v}}$ satisfies $\mathcal{C}_v$ and a class $z_p\in H^1(\op{G}_{\Q_p}, \g)$ such that $(\operatorname{Id}+p^mz_p){\rho_{m+1}}_{\restriction \op{G}_{\Q_p}}$ satisfies $\mathcal{C}_p^{\lambda}$. Since we assume that $m\geq \lambda+2$, we have that $\mathcal{N}_p$ stabilizes $\mathcal{C}_p^{\lambda}$. For $v\in S\cup X$, the elements $z_v$ are defined modulo $\mathcal{N}_v$. Since $\operatorname{res}_{S\cup X}$ is surjective, the tuple
      $(z_v)\in \bigoplus_{v\in S\cup X} H^1(\op{G}_{\Q_v}, \g)/\mathcal{N}_v$ arises from a unique global cohomology class $z$ which is unramified outside $S\cup X$. After replacing $\rho_{m+1}$ by $(\op{Id}+p^m z) \rho_{m+1}$, it satisfies the conditions $\mathcal{C}_v$ at each prime $v\in (S\backslash\{p\})\cup X$ and $\mathcal{C}_p^{\lambda}$ at $p$. This completes the inductive lifting argument. Thus, there is a characteristic zero lift $\rho$ unramified away from $S\cup X$, satisfying $\mathcal{C}_p^{\lambda}$ at $p$. It follows from the main theorem of \cite{KisinFM} that $\rho$ arises from a normalized eigencuspform $f=\sum_{n\geq 1} a_n q^n\in  S_2(\Gamma_1(N))$. Since $\rho$ is crystalline, the level $N$ is prime to $p$. Let $\mathfrak{p}$ be the prime above $p$ in $\Q(f)$ so that $\rho=\rho_{f,\mathfrak{p}}$. Since $\rho$ satisfies $\mathcal{C}_p^{\lambda}$ at $p$, we have that
     \[v_p(\iota_{\mathfrak{p}}(a_p))=\lambda\] (see condition $\eqref{cplambda1}$ on p. 6).
\par In section $\eqref{section3}$, we recall some facts about Fontaine-Laffaille modules. In section $\eqref{section4}$, we define the subfunctor $\mathcal{C}_p^{\lambda}$ and show that it satisfies aforementioned properties. In section $\eqref{lastsection}$, we show that there is a deformation $\rho_{\lambda+2}$ satisfying the aforementioned properties. It is in this section that the main theorem is proved. The arguments in this section are referred to in the proof of Theorem $\ref{main}$.

\section{Recollections on the Fontaine-Laffaille Functor}\label{section3}
\par Let $\op{W}(\F)$ denote the ring of Witt-vectors with residue field $\F$ and let $K$ denote its fraction field $\op{W}(\F)[p^{-1}]$. Let $\sigma\in \Gal(K/\Q_p)$ be the Frobenius element and let $\C_p$ denote the completion of $\bar{\Q}_p$ w.r.t the valuation $v_p$. Recall the notion of a (filtered) $\varphi$-module.
\begin{Def}
A filtered $\varphi$-module $M$ is a $K$ vector space equipped with a semilinear bijective map $\varphi$ with respect to $\sigma$ and a decreasing filtration $\{F^i M\}$. For this filtration, $F^i M=M$ for $i\ll 0$ and $F^i M=0$ for $i\gg 0$.
\end{Def}
The $p$-adic Galois representations that arise from the \'etale cohomology of varieties over $\Q$ satisfy some natural conditions at $p$. From this perspective, there are two predominant conditions at $p$, the ordinary condition and the crystalline condition. The ordinary condition may be described in concrete representation theoretic terms. In this paper, we assume that $\bar{\rho}_{\restriction \op{G}_{\Q_p}}$ is irreducible, and thus we are not in the ordinary setup. The crystalline condition on the other hand is rather subtle, as these are the representations that arise from the Galois action on the generic fibre of a finite flat group scheme over $\Z_p$. The crystalline period functor associates a $\varphi$-module to a crystalline representation $\varrho:\op{G}_{\Q_p}\rightarrow \op{GL}_n(K)$. This allows one to classify crystalline Galois representations by classifying the associated $\varphi$-modules. The Fontaine-Laffaille functor on the other hand will allow us to parametrize flat deformations of $\bar{\rho}_{\restriction \op{G}_{\Q_p}}$. This approach helps reduce the study of finite flat group schemes to linear algebra in the category of Fontaine-Laffaille modules. Ramakrishna used the properties of the Fontaine-Laffaille functor to explicitly characterize flat deformation rings. In this section, we review some properties of the Fontaine-Laffaille functor and make preparations to describe a fixed slope subfunctor of the functor of flat deformations of $\bar{\rho}_{\restriction \text{G}_{\Q_p}}$. The standard references for this section are \cite{FontaineLaffaille, Ram1, patrikisphd}. For an introduction to the flat deformation functor, see \cite{conradflat}.
\par Let us recall the construction of Fontaine's crystalline period ring $\text{B}_{\rm{cris}}$ and the crystalline period functor $\text{D}_{\rm{cris}}$. Set $\Of$ to denote the characteristic-$p$ ring given by the inverse limit $\varprojlim\mathcal{O}_{\C_p}/p$ w.r.t the $p$-power maps $x\mapsto x^p$. We choose a distinguished element $\varepsilon=(\varepsilon_i)$ in $\Of$ for which $\varepsilon_0$ is a primitive $p$-th root of unity. For $x\in \Of$, denote by $[x]$ the Teichm\"uller lift of $x$ in $\op{W}(\Of)$. The continuous Galois-equivariant ring homomorphism
$\theta: \op{W}(\Of)\rightarrow \mathcal{O}_{\C_p}$, defined by $\theta\left(\sum_i [x^{(i)}] p^i\right):=\sum_i x^{(i)}_0 p^i$, is open and surjective. The kernel of $\theta$ is a principal ideal generated by a special element $\xi\in \op{W}(\Of)$ with some key properties stated in \cite[Proposition 4.4.3]{ConradBrinon}. Let $\text{A}_{\rm{cris}}$ denote the divided power envelope of $\op{W}(\Of)$ with respect to the ideal $\ker \theta$. More explicitly, it is given by $\text{A}_{\rm{cris}}=\op{W}(\Of)[\xi^m/m!]_{m\geq 1}$. The $p$-adic completion $\B^+$ is a local ring with residue field $\C_p$ and the crystalline period ring $\B$ is then defined to be the ring obtained on inverting the period $t:=\log[\varepsilon]\in \text{A}_{\rm{cris}}$. The period ring $\B:=\B^+[1/t]$ has an induced Galois stable filtration of $\Q_p$ vector spaces given by $F^i\B:=t^i \text{A}_{\rm{cris}}$.
\par The category of finite dimensional continuous $\op{G}_{\Q_p}$-representations over $K$ is denoted $\RepQp$.
\begin{Def}\label{crysdef}
 Fontaine's crystalline period $\D$ is a functor from $\RepQp$ to the category of finitely filtered $\p$-modules. For $V\in \RepQp$, the module $\D(V)$ is given by
\[\D(V):=\left(V\otimes_{\Q_p} \B\right)^{\op{G}_{\Q_p}}\]
with filtration
\[F^i \D(V):=\left(V\otimes_{\Q_p} F^i\B\right)^{\op{G}_{\Q_p}}.\] The representation $V$ is \textit{crystalline} if \[\dim_K V=\dim_K \text{D}_{\text{\rm{cris}}}(V).\]
\end{Def}
In order to study flat deformations of $\bar{\rho}_{\restriction \op{G}_{\Q_p}}$ over finite length local algebras over $\op{W}(\F)$, we are led to consider Fontaine-Laffaille modules.
\begin{Def}
A Fontaine-Laffaille module $M$ is a finitely generated $\op{W}(\F)$-module that is furnished with a decreasing, exhaustive, separated filtration of $\op{W}(\F)$-submodules $\{F^i M\}$ and for each integer $i$. Furthermore, for each $i$, there is a $\sigma$-semilinear map \[\varphi^i=\varphi_M^i:F^i M\rightarrow M.\] Furthermore, the following conditions are satisfied:
\begin{enumerate}
\item there exists $j_0\geq 0$ such that $F^{-j}M=M$ and $F^j M=0$ for all $j\geq j_0$,
    \item $\varphi^{i+1}=p \varphi^i$,
    \item $\sum_i \varphi^i(F^i M)=M$.
\end{enumerate}
Denote by $\p$ the map $\p^0:F^0M\rightarrow M$.
\end{Def} A map $f:(M,\p^i_M)\rightarrow (N, \p^i_N)$ of Fontaine-Laffaille modules is a $\op{W}(\F)$-module map such that $f(F^i M)\subseteq F^i N$ and 
$\p^i_N\circ f_{\restriction F^i M}=f\circ \p^i_M$.
\begin{Def}
Let $\text{MF}_{tor}^{f}$ denote the category of finite length Fontaine-Laffaille modules. For integers $a<b$, we let $\text{MF}_{tor}^{f,[a,b]}$ be the full subcategory of $\text{MF}_{tor}^{f}$ whose underlying modules $M$ satisfy $F^a M=M$ and $F^b M=0$.
\end{Def}
Let $\RepZp$ be the category of $\op{W}(\F)[\op{G}_{\Q_p}]$-modules that are of finite length over $\op{W}(\F)$. Fontaine and Laffaille in \cite{FontaineLaffaille} showed that $\text{MF}_{tor}^{f}$ and $\text{MF}_{tor}^{f,[a,b]}$ are abelian categories and defined a contravariant functor \[\text{U}_S:\text{MF}_{tor}^{f,[0,p]}\rightarrow \RepZp.\] Following \cite[section 4]{patrikisphd}, we shall use a covariant version $\text{T}$. The reader may also refer to the unpublished notes of Conrad \cite{conradnotes}, where the properties of $\text{T}$ are catalogued.
\begin{Def}\label{dualFM}
Let $M\in \text{MF}_{tor}^{[a,b]}$, its dual $M^*\in \text{MF}_{tor}^{[1-b,1-a]}$ is the module \[M^*=\Hom_{\op{W}(\F)}(M, K/\op{W}(\F))\] with filtration \[F^i(M^*)=\Hom_{\op{W}(\F)}(M/F^{1-i}M, K/\op{W}(\F)).\]
We proceed to describe the maps $\varphi^i_{M^*}:F^i(M^*)\rightarrow M^*$. Let $f\in F^i(M^*)$ and $m\in M$. Since $M=\sum_{j} \varphi^j_M(F^j M)$, the element $m$ may be represented as a sum $m=\sum_j \varphi^j_M(m_j)$. To prescribe $\varphi^i_{M^*}(f)(m)$ it suffices to define $\varphi^i_{M^*}(f)(\varphi_M^j(m_j))$. These are taken as follows 
\[\varphi^i_{M^*}(f)(\varphi^j_M(m_j))=\begin{cases}0 \text{ for }j>-i;\\
f(p^{-i-j}m_j) \text{ for }j\leq -i.\\
\end{cases}\]
\end{Def}
The reader may check that the maps $\varphi_{M^*}^i$ are well defined. We now describe Tate-twists.
\begin{Def}
For an object $M\in \text{MF}_{tor}^{f,[a,b]}$ and $m\in \Z$, the $m$-fold Tate twist of $M$ is the module $M(m)\in \text{MF}_{tor}^{f,[a+m,b+m]}$ whose underlying module is $M$ and $F^i(M(m))=F^{i-m}(M)$ and $\varphi^i_M=\varphi^{i-m}_{M(m)}$.  
\end{Def}
\begin{Def}
We let
\[\text{T}:\text{MF}_{tor}^{[2-p,1]}\rightarrow \RepZp\]
defined by $\text{T}(M)=\text{U}_S(M^*)$. We define
\[\text{T}_1:\text{MF}_{tor}^{[3-p,2]}\rightarrow \RepZp\]
by $\text{T}_1(M):=\text{U}_S(M^*(1))$. Note that $\text{T}(M)=\text{T}_1(M(1))$.
\end{Def}
We now collect a few facts about the functors $\text{T}$ and $\text{T}_1$. The following facts are proved in \cite[section 3.3]{FontaineLaffaille}. The reader may also refer to the proof of \cite[Fact 4.1]{patrikisphd}, where the arguments are summarized.
\begin{Fact}
The functor $\text{T}$ is full and faithful and as a consequence, so is $\text{T}_1$.
\end{Fact}
\begin{Fact}
For each object $M$ of $\MF$,
\[\text{length}M= \text{length}\text{T}(M).\]
\end{Fact}
\par Let $\mathcal{C}_{\op{W}(\F)}$ be the category of finitely generated noetherian local $\op{W}(\F)$-algebras $R$ with maximal ideal $\mathfrak{m}$ and a prescribed mod $\mathfrak{m}$ isomorphism $R/\mathfrak{m}\xrightarrow{\sim} \F$ and refer to $\mathcal{C}_{\op{W}(\F)}$ as the category of \textit{coefficient rings} over $\op{W}(\F)$. A map $f:(R_1, \mathfrak{m}_1)\rightarrow (R_2, \mathfrak{m}_2)$ in $\mathcal{C}_{\op{W}(\F)}$ is a map of local rings compatible with reduction isomorphisms. Let $\mathcal{C}_{\op{W}(\F)}^f$ be the full subcategory of finite length algebras. 

\begin{Def}\label{flatdef}Let $R$ be a finite length coefficient ring and $\varrho:\op{G}_{\Q_p}\rightarrow \op{GL}_2(R)$ be a continuous representation. Then $\varrho$ is said to be \textit{flat} if it arises from the $\op{G}_{\Q_p}$-action of the generic fibre of a finite flat group scheme over $\Z_p$. For $R\in \mathcal{C}_{\op{W}(\F)}$, a continuous representation $\varrho:\op{G}_{\Q_p}\rightarrow \op{GL}_2(R)$ is flat if $\varrho\otimes_R R'$ is flat for all finite length quotients $R'$ of $R$. Define the functor
\[\mathcal{C}_p:\mathcal{C}_{\op{W}(\F)}\rightarrow \text{Sets}\]on $\mathcal{C}_{\op{W}(\F)}$ so that $\mathcal{C}_p(R)$ consists of all deformations $\varrho:\text{G}_{\Q_p}\rightarrow \text{GL}_2(R)$ of $\bar{\rho}_{\restriction \op{G}_{\Q_p}}$
that are flat with determinant $\det\varrho=\chi$.
\end{Def}
\begin{Remark}If $E_{/\Q_p}$ is an elliptic curve with good reduction, then for all $m\geq 1$, the $\op{G}_{\Q_p}$-representation on the $p^m$-torsion points $E[p^m]$ is flat with determinant $\chi$, cf. \cite[Theorem 1.2]{conradflat}. If $E$ has supersingular reduction, then the residual representation $E[p]$ is irreducible as a $\op{G}_{\Q_p}$-module.
\end{Remark}
The following result of Fontaine and Laffaille is the key to working with flat deformations.
\begin{Th}(Fontaine-Laffaille \cite[section 9]{FontaineLaffaille}) 
 Let $R\in\mathcal{C}_{\op{W}(\F)}^f$ and $\varrho\in\op{Def}_p(R)$. Then $\varrho\in \mathcal{C}_p(R)$ if and only if $\varrho=\text{T}_1(M)$ for some $M\in \text{MF}_{tor}^{[0,2]}$.
\end{Th}
Thus, we may explicitly work with modules $M\in \text{MF}_{tor}^{[0,2]}$. As $\text{T}_1$ is full and faithful, the $R$-action on the free $R$-module $\text{T}_1(M)$ carries over to a faithful $R$-action on $M$. In greater detail, if $r\in R$, the endomorphism induced by multiplication by $r$ on $\text{T}_1(M)$ is an $R[\text{G}_{\Q_p}]$-module endomorphism. Also since $\text{T}_1$ is full and faithful, this multiplication by $r$ endomorphism concides with an endomorphism
\[\text{T}_1^{-1}(r):M\rightarrow M\] in $\text{MF}_{tor}^{f}$, i.e.,
\begin{itemize}
\item $\text{T}_1^{-1}(r):M\rightarrow M$ is a $\op{W}(\F)$-module endomorphism
\item $\text{T}_1^{-1}(r)(F^iM)\subseteq F^iM$
\item $\p^i\circ \text{T}_1^{-1}(r)_{\restriction F^iM}=\text{T}_1^{-1}(r)\circ \p^i.$
\end{itemize}
Denote the $\op{W}(\F)$-algebra of $\p$-equivariant $\op{W}(\F)$-module homomorphisms of $M$ by $\text{End}_{\p}(M)$. As a consequence of the faithfulness of $\text{T}_1$, we obtain a ring homomorphism 
\[\text{T}_1^{-1}: R\rightarrow \text{End}_{\p}(M)\]
taking $r\in R$ to the endomorphism $\text{T}_1^{-1}(r)$. We summarize these observations by saying that $M$ is an $R$-module in the category $\text{MF}_{tor}^f$. In particular, it is an $R[\p]$-module and that the $R$-action on $M$ is uniquely determined by that on $\text{T}_1(M)$ and preserves the filtration on $M$. The following result \cite[Lemma 4.2]{patrikisphd} is stated for the functor $\text{T}$, the same argument applies verbatim to $\text{T}_1$.
\begin{Prop}\label{freemodprop}
Let $R\in \mathcal{C}_{\op{W}(\F)}^f$ and $M\in \text{MF}_{tor}^{[0,2]}$. Suppose that $\text{T}_1(M)\in \mathcal{C}_p(R)$. The $R$-module structure on $\text{T}_1(M)$ induces an $R$-module structure on $M$. Furthermore, $M$ is also a free $R$-module of rank $2$ and $\p:M\rightarrow M$ is an $R$-linear endomorphism.
\end{Prop}
Let $R\in \mathcal{C}_{\op{W}(\F)}^f$ and $M\in \text{MF}_{tor}^{[0,2]}$ with $\text{T}_1(M)\in \mathcal{C}_p(R)$. Note that by Proposition $\ref{freemodprop}$, the $R$-module $M$ is free of rank $2$. The $R$-linear operator $\p$ thus may be described by a $2\times 2$ matrix with entries in $R$. The trace of $\p$ is the sum of the diagonal elements of this matrix. We denote the trace by $\t_{\restriction M}$, or sometimes to ease notation when working with Galois representations we use $\t_{\restriction \text{T}_1(M)}$ interchangeably.
\begin{Fact}\label{TRequalsR} Let $R\in \mathcal{C}_{\op{W}(\F)}^f$, we treat $R$ as an object of $\text{MF}_{tor}^f$ which is concentrated in degree zero and $\varphi$ is the identity. Then $\text{T}(R)=R$ with trivial Galois action (see \cite[Lemma 4.6]{patrikisphd}).
\end{Fact}
\begin{Def}
Let $R\in \mathcal{C}_{\op{W}(\F)}^f$ and $A,B\in \text{MF}_{tor}^f$ be $R$-modules such that the $R$-module action is compatible with the $\op{W}(\F)$-action, the filtrations and $\varphi^i$-maps. The tensor product $A\otimes_R B$ has a natural filtration $F^m(A\otimes_R B)=\bigoplus_{i+j=m} F^i(A)\otimes_R F^j (B)$ and $\varphi^m_{A\otimes_R B}=\bigoplus_{i+j=m} \varphi^i_A \otimes_R \varphi^j_B$.
\end{Def}
\begin{Lemma}\label{Tensor}
Let $A\in \text{MF}_{tor}^{f,[0,2]}$ and $B\in \text{MF}_{tor}^{f,[0,1]}$ and $R\in \mathcal{C}_{\op{W}(\F)}^f$ such that $\text{T}_1(A)$ and $\text{T}(B)$ have the structure of $R[\text{G}_{\Q_p}]$-modules compatible with their $\op{W}(\F)[\text{G}_{\Q_p}]$-module structure. Then, there is an isomorphism of $R[\text{G}_{\Q_p}]$-modules $\text{T}_1(A)\otimes_{R} \text{T}(B)\xrightarrow{\sim} \text{T}_1(A\otimes_R B)$.
\end{Lemma}
\begin{proof} Recall that $\text{T}(M)$ is defined if $M\in \text{MF}_{tor}^{f,[2-p,1]}$. The reader is referred to \cite[Lemma 4.3]{patrikisphd} or \cite[Lemma 4.18]{booher} where it is shown that $\text{T}(C)\otimes_{R} \text{T}(D)\xrightarrow{\sim} \text{T}(C\otimes_R D)$ when $\text{T}(C), \text{T}(D)$ and $ \text{T}(C\otimes_R D)$ are defined. Note that $\text{T}_1(A)=\text{T}(A(-1))$. Setting $C=A(-1)\in \text{MF}_{tor}^{f,[-1,1]}$ and $D=B$, we have that $C\otimes_R D\in \text{MF}_{tor}^{f,[-1,1]}$. Since $p\geq 3$, it follows that $\text{T}(C\otimes_R D)$ is defined. Therefore, we have that 
\[\text{T}_1(A)\otimes_{R} \text{T}(B)\xrightarrow{\sim}\text{T}(C)\otimes_{R} \text{T}(D)\xrightarrow{\sim} \text{T}(C\otimes_R D)\xrightarrow{\sim} \text{T}_1(A\otimes_R B).\]
\end{proof}
Let $g:R_1\rightarrow R_2$ be a map of coefficient rings in $\mathcal{C}_{\op{W}(\F)}^f$ and $\text{T}_1(M_1)\in \mathcal{C}_p(R_1)$. The push-forward of $\text{T}_1(M_1)$ to $\mathcal{C}_p(R_2)$ is $g_* \text{T}_1(M_1):=\text{T}_1(M_1)\otimes_{R_1} R_2$. By Fact $\ref{TRequalsR}$, we have that $T(R_2)=R_2$. It follows from Lemma $\ref{Tensor}$ that there is an $R_2[\text{G}_{\Q_p}]$-module isomorphism $\text{T}_1(M_1)\otimes_{R_1} R_2\simeq \text{T}_1(M_1\otimes_{R_1} R_2)$ and as a result, we may identify $g_*\text{T}_1(M_1)$ with $\text{T}_1(M_2)$ where $M_2:=M_1\otimes_{R_1} R_2$.
The following relation comes as a consequence of this identification.
\begin{Prop}Let $g:R_1\rightarrow R_2$ be a map of coefficient rings in $\mathcal{C}_{\op{W}(\F)}^f$. With respect to notation introduced above, for $\varrho\in \mathcal{C}_p(R_1)$, the following relation holds
\begin{equation}\label{relation}\t_{\restriction g_*\varrho}=g\left(\t_{\restriction \varrho}\right).\end{equation}
\end{Prop}
\section{The Fixed-Slope Functor}\label{section4}
\par We recall that $K=\op{W}(\F)[1/p]$ with Frobenius $\sigma$ and $N>0$ coprime to $p$. Let $f\in S_2(\Gamma_1(N))$ be a normalized eigenform with nebentype character $\psi$. Let $\mathfrak{p}|p$ in the number field generated by the Fourier coefficients $\Q(f)$ of $f$. The $\mathfrak{p}$-adic Galois representation $\rho_{f,\mathfrak{p}}:\text{G}_{\Q}\rightarrow \text{GL}_2(\mathcal{O}_{f,\mathfrak{p}})$ where $\mathcal{O}_{f,\mathfrak{p}}$ is the valuation ring of $\Q(f)_{\mathfrak{p}}$ with residue field $k(\mathfrak{p})$. Assume that the following conditions are satisfied
\begin{enumerate}
    \item $\rho_{f,\mathfrak{p}}$ is supersingular, i.e., $f$ is supersingular at $\mathfrak{p}$,
     \item $\Q(f)_{\mathfrak{p}}$ embeds in $K$, or in other words, the residue field $k(\mathfrak{p})\subseteq \F$ and the ramification index of $e(\mathfrak{p}|p)=1$,
    \item $\rho_{f,\mathfrak{p}}$ lifts $\bar{\rho}$.
\end{enumerate}
In addition to these conditions, since $N$ is coprime to $p$ it follows that $\rho_{f,\mathfrak{p}}$ satisifes the flat deformation condition.
We have not made any assumptions on the slope of $\iota_{\mathfrak{p}}(a_p(f))$ except that it is positive. Let $V_f$ denote the $2$-dimensional $\mathfrak{p}$-adic $\op{G}_{\Q_p}$-representation induced by $f$. This representation is flat and in particular, crystalline and hence,
\[\dim_{K} \D(V_f)=\dim_{K} V_f=2.\]
The following result is \cite[Theorem 1.2.4 (ii)]{scholl}. 
\begin{Th} \label{Saito}
 The characteristic polynomial of $\p$ on the $2$ dimensional $K$ vector space $\D(V_f)$ is
\[ch(X)=X^2-\iota_{\mathfrak{p}}(a_p) X+\psi(p)p.\]
\end{Th}
\par In order to describe a \textit{ fixed slope} deformation problem we proceed to describe the notion of the \textit{slope} of any flat deformation of $\bar{\rho}_{\restriction\text{G}_{\Q_p}}$ to any Artinian coefficient ring $R$
\[ \begin{tikzpicture}[node distance = 2.5 cm, auto]
            \node at (0,0) (G) {$\text{G}_{\Q_p}$};
             \node (A) at (3,0){$\text{GL}_2(\F)$.};
             \node (B) at (3,2){$\text{GL}_2(R)$};
      \draw[->] (G) to node [swap]{$\bar{\rho}_{\restriction \text{G}_{\Q_p}}$} (A);
       \draw[->] (B) to node{} (A);
      \draw[->] (G) to node {$\varrho$} (B);
      \end{tikzpicture}\]
      The following result gives us control on the $p$-adic valuation of $\iota_{\mathfrak{p}}(a_p(f))$ via torsion Fontaine-Laffaille theory.
\begin{Prop}\label{comparisonprop}
Let $\rho:\op{G}_{\Q_p}\rightarrow \op{GL}_2(\op{W}(\F)/p^m)$ be a flat deformation of $\bar{\rho}_{\restriction \op{G}_{\Q_p}}$ such that there exists a cuspidal Hecke eigenform $f\in S_2(\Gamma_1(N))$ of level $N$ coprime to $p$ such that $\mathcal{O}_{f,\mathfrak{p}}=\op{W}(\F)$ and $\rho_{f\restriction \op{G}_{\Q_p}}$ lifts $\rho$
\[ \begin{tikzpicture}[node distance = 2.5 cm, auto]
            \node at (0,0) (G) {$\op{G}_{\Q_p}$};
             \node (A) at (3,0){$\op{GL}_2(\op{W}(\F)/p^m)$.};
             \node (B) at (3,2){$\op{GL}_2(\op{W}(\F))$};
      \draw[->] (G) to node [swap]{$\rho$} (A);
       \draw[->] (B) to node{} (A);
      \draw[->] (G) to node {$\rho_{f\restriction \op{G}_{\Q_p}}$} (B);
      \end{tikzpicture}\] Then $\t_{\restriction \rho}=\iota_{\mathfrak{p}}(a_p(f))\mod{p^m}$.
\end{Prop}
\begin{proof}
Let $V_f=V_{f,\mathfrak{p}}$ be the underlying vector space on which $\text{G}_{\Q_p}$ acts via $\rho_{f,\mathfrak{p}}$. There is a Galois stable $\op{W}(\F)$-lattice $L$ in $V_f$ on which the Galois group acts via the integral representation \[\rho_{f,\mathfrak{p}}:\op{G}_{\Q}\rightarrow \op{GL}_2(\op{W}(\F)).\] The module $\mathcal{M}:=\varprojlim_n \text{T}_1^{-1}(L/p^n L)$ is equipped with a $\varphi$-action which is by definition the inverse limit of the $\varphi$-actions on $\text{T}_1^{-1}(L/p^n L)$ for $n\geq 1$. It follows from Proposition $\ref{freemodprop}$ that $\mathcal{M}$ is a free module of rank $2$ over $\op{W}(\F)$. Let $\text{MF}=\text{MF}_{\op{W}(\F)}$ be the category of Fontaine-Laffaille modules over $\op{W}(\F)$. Let $\text{A}_{\rm{cris},\infty}$ be the direct limit $\text{A}_{\rm{cris},\infty}:= \varinjlim_n \text{A}_{\rm{cris}}/p^n \text{A}_{\rm{cris}}$. The $\op{W}(\F)$-modules $\text{A}_{\rm{cris}}$ and $\text{A}_{\rm{cris},\infty}$ are Fontaine-Laffaille modules, see \cite[p. 9]{hattori2018integral}. Following \textit{loc. cit.} we denote by \[\begin{split}&\text{T}^*_{\rm{cris}}(\mathcal{M}):=\Hom_{\text{MF}}(\mathcal{M},\text{A}_{\rm{cris},\infty}),\\
&\hat{\text{T}}^*_{\rm{cris}}(\mathcal{M}):=\Hom_{\text{MF}}(\mathcal{M},\text{A}_{\rm{cris}}).
\end{split}\] It will be shown that there is a $\varphi$-equivariant isomorphism of $K$-vector spaces \begin{equation}\label{hardtoprove}\text{D}_{\rm{cris}}(V_f)\simeq \mathcal{M}\otimes_{\op{W}(\F)} K.\end{equation} Let us explain how the result follows once the above isomorphism is established. Note that the action of $\op{G}_{\Q_p}$ on $L/p^m L$ is via $\rho$. Hence, by definition, $\t_{\restriction \rho}$ is the trace of $\varphi$ on $\op{T}_1^{-1}(L/p^mL)$. On the other hand, it follows from \eqref{hardtoprove} and Theorem $\ref{Saito}$, that the trace of $\varphi$ acting on $\mathcal{M}$ is equal to $\iota_{\mathfrak{p}}(a_p(f))$. Since $\op{T}_1^{-1}(L/p^mL)=\mathcal{M}/p^m$, it follows that \[\t_{\restriction \rho}=\iota_{\mathfrak{p}}(a_p(f))\mod{p^m}.\]
\par We now prove $\eqref{hardtoprove}$. Let $V$ denote the two dimensional $K$-vector space \[V:=\hat{\text{T}}_{\rm{cris}}^*(\mathcal{M})\otimes_{\op{W}(\F)} K\] with $\op{G}_{\Q_p}$-action. According to \cite[Proposition 3]{bre99}, the functor $\hat{\text{T}}^*_{\rm{cris}}$ induces an  anti-equivalence between the category of \textit{strongly divisible lattices} in $\op{D}_{\rm{cris}}(V^*)$ and that of
$\op{G}_{\Q_p}$-stable lattices in $V$. The reader may also refer to \cite[Theorem 2.11]{hattori2018integral}. We only need to note that there is a $\varphi$-equivariant isomorphism of $K$-vector spaces\[\text{D}_{\rm{cris}}(V^*)\simeq \mathcal{M}\otimes_{\op{W}(\F)} K.\] Thus to complete the proof of the result, it suffices to show that there is an $\op{G}_{\Q_p}$-equivariant isomorphism of $K$-vector spaces 
\begin{equation}
V_f\simeq V^*.
\end{equation}
Note that $V_f=L\otimes_{\op{W}(\F)}K$ and $V^*=\Hom_{\op{W}(\F)}(\hat{\text{T}}^*_{\rm{cris}}(\mathcal{M}),\op{W}(\F)(1))\otimes_{\op{W}(\F)}K$.
Thus, it suffices to show that there is a $\text{G}_{\Q_p}$-equivariant isomorphism of $K$-vector spaces \begin{equation}\label{technicaliso}L\otimes_{\op{W}(\F)}K\simeq \Hom_{\op{W}(\F)}(\hat{\text{T}}^*_{\rm{cris}}(\mathcal{M}),\op{W}(\F)(1))\otimes_{\op{W}(\F)}K.\end{equation}
\par Following \cite{hattori2018integral} $\text{U}_S\simeq \text{T}_{\rm{cris}}^*$ where $\text{T}_{\rm{cris}}^*$ is prescribed by 
\[\text{T}_{\rm{cris}}^*(M)=\Hom_{\text{MF}}(M,\varinjlim_{n}\text{A}_{\rm{cris}}/p^n),\] and thus,
\begin{equation}\label{T1equality}\begin{split}\text{T}_1(M)=&\Hom_{\op{W}(\F)}(\text{T}_{\rm{cris}}^*(M),K/\op{W}(\F)(1))\\
=& \Hom_{\op{W}(\F)}(\Hom_{\text{MF}}(M,\varinjlim_{n}\text{A}_{\rm{cris}}/p^n),K/\op{W}(\F)(1)).\end{split}\end{equation}We now proceed to prove $\eqref{technicaliso}$ by simplifying both sides of the isomorphism. We observe that
\[\begin{split}
&\Hom_{\op{W}(\F)}(\hat{\text{T}}^*_{\rm{cris}}(\mathcal{M}),\op{W}(\F)(1))\\
=&\Hom_{\op{W}(\F)}(\Hom_{MF}(\mathcal{M},\text{A}_{\rm{cris}}),\op{W}(\F)(1))\\
\simeq &\Hom_{\op{W}(\F)}(\Hom_{MF}(\mathcal{M},\text{A}_{\rm{cris}})\otimes_{\op{W}(\F)}K/\op{W}(\F),K/\op{W}(\F)(1))\\
\simeq &\Hom_{\op{W}(\F)}(\varinjlim_{n}\Hom_{MF}(\mathcal{M},\text{A}_{\rm{cris}})/p^n,K/\op{W}(\F)(1)).\\
\end{split}
\]
On the other hand, 
\[\begin{split}
L\simeq & \varprojlim_m\text{T}_1(\mathcal{M}/p^m )\\
\simeq & \varprojlim_m\Hom_{\op{W}(\F)}(\Hom_{MF}(\mathcal{M}/p^m,\varinjlim_{n}\text{A}_{\rm{cris}}/p^n),K/\op{W}(\F)(1))\\
\simeq & \Hom_{\op{W}(\F)}(\varinjlim_m\Hom_{MF}(\mathcal{M}/p^m,\varinjlim_{n}\text{A}_{\rm{cris}}/p^n),K/\op{W}(\F)(1))\\
\simeq & \Hom_{\op{W}(\F)}(\varinjlim_m\Hom_{MF}(\mathcal{M},\text{A}_{\rm{cris}}/p^m),K/\op{W}(\F)(1)).
\end{split}\]
The $\text{G}_{\Q_p}$-equivariant inclusion
\[\varinjlim_{n}\Hom_{MF}(\mathcal{M},\text{A}_{\rm{cris}})/p^n
\hookrightarrow \varinjlim_n\Hom_{MF}(\mathcal{M},\text{A}_{\rm{cris}}/p^n)
\] 
induces a $\text{G}_{\Q_p}$-equivariant surjection \[L\otimes_{\op{W}(\F)}K\twoheadrightarrow \Hom_{\op{W}(\F)}(\hat{\text{T}}^*_{\rm{cris}}(\mathcal{M}),\op{W}(\F)(1))\otimes_{\op{W}(\F)}K\simeq V^*.\]
On the other hand, we observe that \[\dim_K L\otimes_{\op{W}(\F)}K=\text{rank}_{\op{W}(\F)} L=2 \]
and \[\dim_K V^*\geq \dim_K \text{D}_{\rm{cris}}(V^*)=\dim_K \mathcal{M}\otimes_{\op{W}(\F)}K=2,\] and thus this surjection must be an isomorphism.
\end{proof}
\begin{Def}
Recall that $\mathcal{C}_p$ is the flat deformation functor (with fixed determinant). For $R\in \mathcal{C}_{\op{W}(\F)}^f$ with maximal ideal $\mathfrak{m}_R$ and $\lambda\in \Z_{\geq 1}$ we define \[\Fl(R)\subseteq (\mathcal{C}_p\times \hat{\mathbb{A}}^1)(R)= \mathcal{C}_p(R)\times \mathfrak{m}_R \] the set of pairs $(\rho,U)$ such that \[\t_{\restriction \rho}=p^{\lambda}(1+U).\] That $\mathcal{F}^{\lambda}$ is a functor follows from $\eqref{relation}$.
\end{Def}
Note that the functor $\mathcal{F}^{\lambda}$ does not account for all pairs $(\rho,V)\in \mathcal{C}_p(R)\times \mathfrak{m}_R$ such that $V=\t_{\restriction \rho}$ is $p^{\lambda}$ times a unit in $R$. For the purpose of proving the lifting theorem, it suffices to consider only traces which are expressible as $p^{\lambda}$ times a unit which is an element of $(1+\mathfrak{m}_R)$. This choice is made for convenience. The second component of $(\mathcal{C}_p\times \hat{\mathbb{A}}^1)(R)$ parametrizes principal units in $R$. The flat deformation ring at $p$ is a $2$-variable power series ring $\op{W}(\F)[[X_1,X_2]]$ (see \cite[Theorem 3.1]{Ram1}). Recall that the functor $\mathcal{C}_p$ parametrizes flat deformations with determinant $\chi_{\restriction\op{G}_{\Q_p}}$. The dimension of the tangent space $\mathcal{N}_p$ is equal to $1$ (see \cite[Table 1 p. 125]{RaviFM}) and $\mathcal{C}_p$ is pro-represented by a power-series ring in a single variable $\mathcal{R}:=\op{W}(\F)[[X]]$. The functor $\mathcal{C}_p\times\hat{\mathbb{A}}^1$ is pro-represented by the fibered product \[\text{Spec}\mathcal{R}\times_{\text{Spec}\op{W}(\F)}\hat{\mathbb{A}}^1\simeq \text{Spec}\op{W}(\F)[[X,Y]].\]
Let $\mathcal{I}$ be an ideal in $\mathcal{R}=\op{W}(\F)[[X]]$ such that $\mathcal{R}\slash \mathcal{I}$ is Artinian. Let $\rho_{\mathcal{R}}$ be the universal representation representing $\mathcal{C}_p$ and let $\rho_{\mathcal{R}\slash \mathcal{I}}$ be the reduction of $\rho_{\mathcal{R}}$ mod $\mathcal{I}$. Let $\mathcal{I}'$ be an ideal which contains $\mathcal{I}$, then by $\eqref{relation}$, we have that
\[\t_{|\rho_{\mathcal{R}\slash \mathcal{I}'}}=\t_{|\rho_{\mathcal{R}\slash \mathcal{I}}}\mod \mathcal{I'}.\]
Let $\Phi(X)\in \mathcal{R}=\op{W}(\F)[[X]]$ be the limit over ideals $\mathcal{I}$ for which $\mathcal{R}\slash \mathcal{I}$ is Artinian
\[\Phi(X):=\varprojlim_{\mathcal{I}} \t_{|\rho_{\mathcal{R}\slash \mathcal{I}}}.\]
\par Letting \[\tilde{\mathcal{R}}^{\lambda}:=\frac{\op{W}(\F)[[X,Y]]}{(\Phi(X)-p^{\lambda}(1+Y))},\] we observe that the equality $\Phi(X)=p^{\lambda}(1+Y)$ corresponds to the condition $\t_{\restriction \rho}=p^{\lambda}(1+U)$. The functor $\mathcal{F}^{\lambda}$ is pro-represented by 
\[X_{\mathcal{F}^{\lambda}}=\text{Spec}\tilde{\mathcal{R}}^{\lambda}=\text{Spec}\left(\frac{\op{W}(\F)[[X,Y]]}{(\Phi(X)-p^{\lambda}(1+Y))}\right).\] Let $\mathfrak{m}_{\mathcal{R}}$ be the maximal ideal of $\mathcal{R}=\op{W}(\F)[[X]]$, we observe that \[\t_{|\bar{\rho}}=\Phi(X)\:\text{mod}\: \mathfrak{m}_{\mathcal{R}}.\] It follows from the theorem of Khare and Wintenberger that $\bar{\rho}=\bar{\rho}_{g,\mathfrak{q}}$ where $g$ is a cuspidal Hecke eigenform and $\mathfrak{q}|p$ is a prime in the field of Fourier coefficients $\Q(g)$. Since $\bar{\rho}_{\restriction \op{G}_{\Q_p}}$ is irreducible, $g$ is supersingular at $\mathfrak{q}$, i.e., $p|\iota_{\mathfrak{q}}(a_p(g))$. It follows from Proposition $\ref{comparisonprop}$ the trace of the residual representation $\t_{|\bar{\rho}}=0$ and as a consequence $\Phi(X)\in \mathfrak{m}_{\mathcal{R}}$. The projection to the second factor is a natural transformation $\pi_2:\mathcal{C}_p\times\hat{\mathbb{A}}^1\rightarrow \hat{\mathbb{A}}^1$ and restricts to a natural transformation $\pi_2: \mathcal{F}^{\lambda}\rightarrow \hat{\mathbb{A}}^1$ which induces the map of $\op{W}(\F)$-algebras
\[\pi_2^*:\op{W}(\F)[[Y]]\rightarrow \tilde{\mathcal{R}}^{\lambda}=\frac{\op{W}(\F)[[X,Y]]}{(\Phi(X)-p^{\lambda}(1+Y))}\] for which $\pi_2^*(Y)=Y\:\text{mod}(\Phi(X)-p^{\lambda}(1+Y))$.
\begin{Lemma}\label{lemma54}
The above map $\pi_2^*$ is an isomorphism if the coefficient of $X$ in the power series expansion of $\Phi(X)$ is a unit in $\op{W}(\F)$.
\end{Lemma}
\begin{proof}
     As explained in the above paragraph, it follows from Proposition $\ref{comparisonprop}$ that $\Phi(X)\in \mathfrak{m}_{\mathcal{R}}$ and therefore the constant coefficient is zero. Note that it follows from the algebraic independence of $X$ and $Y$ in the power series ring $\op{W}(\F)[[X,Y]]$ that $\pi_2^*$ is indeed injective. In order to prove surjectivity, it suffices to show that $X$ is in the image of $\pi_2^*$. In other words, we express $X$ as an element of $\op{W}[[Y]]$, after going modulo the ideal generated by $\Phi(X)-p^{\lambda}(1+Y)$. Write $\Phi(X)$ as a product $cX(1+Xv(X))$, where $c\in\op{W}(\F)$ is a unit and $v(X)\in \op{W}(\F)[[X]]$. Therefore, the relation 
\[\Phi(X)=p^{\lambda}(1+Y)\] may be rephrased as 
\begin{equation}\label{22aug}X=c^{-1}p^{\lambda}(1+Y)\sum_{i=0}^{\infty}(-1)^i X^i v(X)^i.\end{equation}The above relation \eqref{22aug} implies that 
\[X\equiv c^{-1}p^{\lambda}\mod{\mathfrak{m}_{\mathcal{R}}^2}.\] By induction, we show that for all $n\geq 1$, there is a polynomial $f_n(Y)\in \op{W}(\F)[Y]$ such that 
\[X\equiv f_n(Y)\mod{\mathfrak{m}_{\mathcal{R}}^n}.\] Assume this is true for $n$, we show that it is true for $n+1$ as well. Clearly, 
\[X\equiv f_n(Y)+g(X,Y)\mod{\mathfrak{m}_{\mathcal{R}}^{n+1}},\]where \[g(X,Y)=\sum_{0\leq i,j\leq n} a_{i,j} p^i X^j Y^{n+1-i-j}.\] Here, $a_{i,j}=0$ or is a unit in $\op{W}(\F)$. At this stage, plug in the relation \eqref{22aug} and expand all terms of
\[f_n(Y)+g\left(c^{-1}p^{\lambda}(1+Y)\sum_{i=0}^{\infty}(-1)^i X^i v(X)^i, Y\right)\] up to modulo ${\mathfrak{m}_{\mathcal{R}}^{n+1}}$. It is easy to see that going mod-${\mathfrak{m}_{\mathcal{R}}^{n+1}}$, the above expression is represented by a polynomial $f_{n+1}(Y)$ in $Y$ alone. This completes the inductive argument which shows that $X$ is in the image of $\pi^*$. Hence, the map $\pi^*$ is surjective as well, and is therefore an isomorphism.
\end{proof}
\begin{Prop}\label{pi2}
The map $\pi_2^*$ is an isomorphism of $\op{W}(\F)$-algebras and thus $\mathcal{F}^{\lambda}$ is representable by a power series ring in one variable over $\op{W}(\F)$.
\end{Prop}
\begin{proof}
\par By Lemma $\ref{lemma54}$, it suffices to show that the coefficient of $X$ in the power series expansion of $\Phi(X)$ is a unit. Let $M_0\in \text{MF}_{tor}^{f,[0,2]}$ be such that $\text{T}_1(M_0)\simeq \bar{\rho}_{|\op{G}_{\Q_p}}$. Choose a basis $\{e,k\}$ of $M_0$ such that $k$ spans $F^1 M_0$. The Fontaine-Laffaille module $M_0$ is characterized by a single matrix 
\[\left( {\begin{array}{c|c}
   a & c \\
   b & d \\
  \end{array} } \right)\in \text{M}_2(\F),\]
  where
\[\varphi(e)=ae+bk\text{ and }\varphi^1(k)=ce+dk.   \]The relation $\varphi_{\restriction F^1 M_0}=p\varphi^1=0$ implies that $\varphi(k)=0$. Therefore, the matrices corresponding to $\varphi$ and $\varphi^1$ are as follows \[\varphi=\mtx{a}{0}{b}{0}\text{, and }\varphi^1=\left( {\begin{array}{c}
   c \\
   d \\
  \end{array} } \right).\] The trace $a=\t_{|M_0}=0$, since $\bar{\rho}_{|\op{G}_{\Q_p}}$ is supersingular. Let $\tilde{M}_0$ be a Fontaine-Laffaille module of dimension $4$ over $\F$ provided with the structure of a free $\F[\epsilon]$-module which we proceed to describe. We let
$\tilde{M}_0=\F[\epsilon] \cdot\tilde{e}\oplus \F[\epsilon] \cdot\tilde{k}$ with filtration and $\varphi^i$ defined as follows:
\begin{itemize}
\item $F^0\tilde{M}_0:=\tilde{M}_0$, $F^1\tilde{M}_0:=\F[\epsilon]\cdot \tilde{k}$ and $F^2\tilde{M}_0:=0$,
\item the maps $\varphi^j$ for $j=0,1$ are $\F[\epsilon]$-module maps, which similar to $M_0$ is characterized by a matrix
$\left( {\begin{array}{c|c}
   \tilde{a} & \tilde{c} \\
   \tilde{b} & \tilde{d} \\
  \end{array} } \right)\in \text{M}_2(\F[\epsilon])$
  where
\[\varphi(\tilde{e})=\tilde{a}\cdot \tilde{e}+\tilde{b}\cdot \tilde{k} \text{ and }\varphi^1(\tilde{k})=\tilde{c}\cdot \tilde{e}+\tilde{d}\cdot \tilde{k}.   \]
\end{itemize}
We further assume that $\tilde{r}$ lifts $r$ for $r\in \{a,b,c,d\}$, then the map $\mathcal{Q}:\tilde{M}_0\rightarrow M_0$ mapping $\tilde{e}\mapsto e$ and $\tilde{k}\mapsto k$ is a map of Fontaine-Laffaille modules since $\varphi^j\circ\mathcal{Q}=\mathcal{Q}\circ \varphi^j$ for $j=0,1$.
\par It follows from previous discussions that $\text{T}_1(\tilde{M}_0)$ is a rank $2$ Galois representation $\varrho$ which lifts $\bar{\rho}_{|\op{G}_{\Q_p}}$
\[ \begin{tikzpicture}[node distance = 2.5 cm, auto]
            \node at (0,0) (G) {$\text{G}_{\Q_p}$};
             \node (A) at (3,0){$\text{GL}_2(\F)$.};
             \node (B) at (3,2){$\text{GL}_2(\F[\epsilon])$};
      \draw[->] (G) to node [swap]{$\bar{\rho}_{\restriction \text{G}_{p}}$} (A);
       \draw[->] (B) to node{} (A);
      \draw[->] (G) to node {$\varrho$} (B);
      \end{tikzpicture}\]
      We observe that $\t_{|\varrho}=\tilde{a}\in \F[\epsilon]$ lifts $a=0$. We simply specify $\tilde{M}_0$ by choosing the matrix $\left( {\begin{array}{c|c}
   \tilde{a} & \tilde{c} \\
   \tilde{b} & \tilde{d} \\
  \end{array} } \right)$ such that $\tilde{a}\neq 0$, clearly such a choice can be made. The Galois representation $\varrho$ coincides with a $\op{W}(\F)$-algebra homomorphism $H_{\varrho}:\mathcal{R}\rightarrow \F[\epsilon]$ such that 
  \begin{equation}\label{Hvarrho}H_{\varrho}(\Phi(X))=\tilde{a}\neq 0.\end{equation} Express $\Phi(X)=\sum_{i=1}^{\infty} b_i X^i$, it follows from $\eqref{Hvarrho}$ that $b_1$ is a unit. It follows from Lemma $\ref{22aug}$ that $\pi_2^*$ is an isomorphism.
\end{proof}
\begin{Def}
Let $\mathcal{C}_p^{\lambda}$ be the subfunctor of $\mathcal{C}_p$ prescribed by \[\mathcal{C}_p^{\lambda}:=\op{image} \{\mathcal{F}^{\lambda}\rightarrow \mathcal{C}_p\}\] and thus 
\[\mathcal{C}_p^{\lambda}(R):=\{\rho\in \mathcal{C}_p(R)\mid \text{ there exists }U\in \mathfrak{m}_R\text{ such that }\t_{\restriction \rho}=p^{\lambda}(1+U)\}. \]
\end{Def}
     \begin{Prop}
     \label{liftableprop}
     The deformation functor $\mathcal{C}_p^{\lambda}$ is a liftable.
     \end{Prop}
     \begin{proof}
      By Proposition $\ref{pi2}$ it follows that the functor $\mathcal{F}^{\lambda}\simeq \hat{\mathbb{A}}^1$ and thus, $\mathcal{C}_p^{\lambda}$ is a liftable deformation functor.
     \end{proof}
     Recall that $\mathcal{N}_p$ is the tangent space of the flat deformation functor $\mathcal{C}_p$ and that the dimension of the tangent space $\mathcal{N}_p$ is equal to $1$ (see \cite[Table 1 p. 125]{RaviFM}). The following proposition shows that $\mathcal{N}_p$ acts as a tangent space to the subfunctor $\mathcal{C}_p^{\lambda}$ for mod-$p^m$ lifts when $m$ is suitably large. The proposition makes this statement precise.
    \begin{Prop}\label{localatp}
    For $\lambda\in\Z_{\geq 1}$, the space $\mathcal{N}_p$ stabilizes mod $p^m$ lifts of $\bar{\rho}_{\restriction \op{G}_{\Q_p}}$ in $\mathcal{C}_p^{\lambda}$ for $m\geq  \lambda+2$. In other words, for $m\geq  \lambda+2$, $X\in \mathcal{N}_p$ and $\varrho_m\in \mathcal{C}_p^{\lambda}(\op{W}(\F)/p^m)$, the twist
    \[(\op{Id}+p^{m-1}X)\varrho_m\in \mathcal{C}_p^{\lambda}(\op{W}(\F)/p^m).\]
    \end{Prop}
    \begin{proof}
    \par Let $\varrho_{\lambda+1}:\op{G}_{\Q_p}\rightarrow \op{GL}_2(\op{W}(\F)/p^{\lambda+1})$ be the mod-$p^{\lambda+1}$ reduction of $\varrho_m$. Set $\varrho_m'$ to denote the twist $(\op{Id}+p^{m-1} X)\varrho_m$, and $\varrho_{\lambda+1}'$ the mod-$p^{\lambda+1}$ reduction of $\varrho_m'$. Note that $\varrho_m'$ satisfies $\mathcal{C}_p$. Clearly, we have that  $\varrho_{\lambda+1}=\varrho_{\lambda+1}'$. Therefore, $\varrho_{\lambda+1}'$ satisfies $\mathcal{C}_p^{\lambda}$. We show that this is enough to conclude that $\varrho_m'$ satisfies $\mathcal{C}_p^{\lambda}$. Since $\varrho_{\lambda+1}'$ satisfies $\mathcal{C}_p^{\lambda}$, we have that
    \[\t_{\restriction \varrho_{\lambda+1}'}=p^{\lambda}(1+V), \]where $V$ is divisible by $p$. Since $\rho_{\lambda+1}'$ is a mod-$p^{\lambda+1}$ representation, $p^{\lambda} V=0$ and thus,\[\t_{\restriction \varrho_{\lambda+1}'}=p^{\lambda}.\] We deduce that \[\t_{\restriction \varrho_{m}'}=p^{\lambda}\mod{p^{\lambda+1}}\]and hence,  \[\t_{\restriction \varrho_{m}'}=p^{\lambda}(1+U), \]where $U$ is divisible by $p$. This completes the proof.
    \end{proof}
    \section{Lifting to mod $p^{\lambda+2}$ and the proof of the main Theorem}\label{lastsection}
    \par Fix $\lambda\in \Z_{\geq 1}$. In this section, we show that there is a finite set of nice primes $Z$ disjoint from $S$ and a deformation $\rho_{\lambda+2}$
    \[ \begin{tikzpicture}[node distance = 2.5 cm, auto]
            \node at (0,0) (G) {$\op{G}_{\Q,S\cup Z}$};
             \node (A) at (3,0){$\op{GL}_2(\F)$,};
             \node (B) at (3,2){$\op{GL}_2(\op{W}(\F)/p^{\lambda+2})$};
      \draw[->] (G) to node [swap]{$\bar{\rho}$} (A);
       \draw[->] (B) to node{} (A);
      \draw[->] (G) to node {$\rho_{\lambda+2}$} (B);
      \end{tikzpicture}\]
      such that $\rho_{\lambda+2}$ satisfies the conditions $\mathcal{C}_v$ for $v\in (S\backslash\{p\})\cup Z$ and $\mathcal{C}_p^{\lambda}$ at $p$. From here on in, we use the following simplifying notation.
      \begin{Def}\label{stardef}
      We say that a deformation of $\rho_m$ satisfies condition $(\star)$ if satisfies the conditions $\mathcal{C}_v$ at all primes $v\neq p$ at which $\rho_m$ is allowed to ramify and $\mathcal{C}_p^{\lambda}$ at $p$.
      \end{Def}Thus, we wish to show that a lift $\rho_{\lambda+2}$ exists which satisfies $(\star)$. Referring to section $\eqref{sectionoverview}$, this is a crucial ingredient in the proof of Theorem $\ref{main}$. Once we exhibit $\rho_{\lambda+2}$, the arguments in section $\eqref{sectionoverview}$ show that $\rho_{\lambda+2}$ can be lifted to characteristic zero. We shall then be primed to prove Theorem $\ref{main}$.
      \par 
    According to duality of $\Sha$-groups, the group $\Sha^2_S(\text{Ad}^0\bar{\rho})$ is dual to $\Sha^1_S(\text{Ad}^0\bar{\rho}^*)$. Standard arguments show that there is a finite set of primes $V$ disjoint from $S$ such that 
    \begin{itemize}
        \item $V$ consists of nice primes,
        \item $\Sha^1_{S\cup V}(\text{Ad}^0\bar{\rho})$ and $\Sha^2_{S\cup V}(\text{Ad}^0\bar{\rho})$ are zero.
    \end{itemize}
    The reader may refer to \cite[Lemma 6]{KLR}.
    \par We lift $\bar{\rho}$ to $\rho_{\lambda+2}$ via an inductive lifting argument. Let $m \leq \lambda+1$. Assume that there exists a finite set of nice primes $X_m\supseteq V$ disjoint from $S$ and a mod-$p^m$ deformation $\rho_m:\op{G}_{\Q,S\cup X_m}\rightarrow \op{GL}_2(\op{W}(\F)/p^m)$ satisfying $(\star)$. To clarify, this means that $\rho_m$ satisfies conditions $\mathcal{C}_v$ at $v\in (S\backslash \{p\})\cup X_m$ and $\mathcal{C}_p^{\lambda}$ at $p$. Note that since $X_m$ contains $V$, we have that $\Sha^2_{S\cup X_m}(\g)=0$. It follows that $\rho_m$ lifts to a Galois representation $\rho_{m+1}:\op{G}_{\Q, S\cup X_{m}}\rightarrow \op{GL}_2(\op{W}(\F)/p^{m+1})$. This is because the obstruction-class $\mathcal{O}(\rho_m)$ is zero, the argument is identical to that in $\eqref{sectionoverview}$. We seek to find a finite set of nice primes $X_{m+1}$ containing $X_m$ and a cohomology class $z\in H^1(\op{G}_{\Q,S\cup X_{m+1}}, \g)$ such that the twist $\rho_{m+1}=(\op{Id}+p^{m-1} z)\rho_{m+1}'$ satisfies $(\star)$. We then replace $\rho_{m+1}'$ by $\rho_{m+1}$. This shall complete the inductive lifting argument, and we shall deduce that there is a deformation $\rho_{\lambda+2}:\op{G}_{\Q, S\cup X_{\lambda+2}} \rightarrow \op{GL}_2(\op{W}(\F)/p^{\lambda+2})$ satisfying $(\star)$. We shall set $Z:=X_{\lambda+2}$.
    
      \begin{Lemma}\label{closetotheend}
      Fix $\lambda\in \Z_{\geq 1}$ and let $m\leq \lambda+1$. Assume that there a finite set of nice primes $X_m$ containing $V$ and a deformation $\rho_m:\op{G}_{\Q,S\cup X_m}\rightarrow \op{GL}_2(\op{W}(\F)/p^m)$ satisfying $(\star)$. There exists a finite choice of nice primes $X_{m+1}$ containing $X_m$ and a deformation $\rho_{m+1}:\op{G}_{\Q,S\cup X_{m+1}}\rightarrow \op{GL}_2(\op{W}(\F)/p^{m+1})$ of $\rho_m$ which satisfies $(\star)$.
      \end{Lemma}
      \begin{proof}
    The argument is similar to that of \cite[pp. 725-726]{KLR}.
        Since $X_m$ contains $V$, we have that $\Sha^2_{S\cup X_m} (\g)$ is zero. Thus, since $\rho_{m}$ satisfies a liftable deformation functor at each prime at which it is allowed to ramify, there is a global lift $\rho_{m+1}:\op{G}_{\Q,S\cup X_{m}}\rightarrow \op{GL}_2(\op{W}(\F)/p^{m+1})$. This is because the obstruction-class $\mathcal{O}(\rho_m)$ is zero, the argument is identical to that in $\eqref{sectionoverview}$.
        According to \cite[Lemma 8]{KLR} there is a finite set of nice primes $X_{m+1}'$ containing $X_m$ such that 
    \begin{enumerate}
        \item $\rho_m$ satisfies $\mathcal{C}_v$ at each prime $v\in X_{m+1}'/X_m$.
        \item The restriction map
        \begin{equation}\label{isolast}H^1(\op{G}_{\Q,S\cup X_{m+1}'}, \g)\rightarrow \bigoplus_{v\in S\cup X_m} H^1(\op{G}_{\Q_v}, \g)
          \end{equation}
          is an isomorphism.
    \end{enumerate}
        \par Let $v$ be a prime in $S\cup X_m$. There is a local cohomology class 
        \[z_v\in H^1(\op{G}_{\Q_v}, \g)\]
        such that the twist $(\op{Id}+p^{m} z_v)\rho_{m+1}$ satisfies $\mathcal{C}_v$ for $v\neq p$ and $\mathcal{C}_p^{\lambda}$ for $v=p$. Since the map $\eqref{isolast}$ is an isomorphism, there exists $z\in H^1(\op{G}_{\Q, S\cup X_{m+1}'}, \g)$ such that $z_{\restriction \op{G}_{\Q_v}}=z_v$ for each prime $v\in S\cup X_{m}$. Thus, by construction, the twist $\rho_{m+1}':=(\op{Id}+p^m z) \rho_{m+1}$ satisfies $\mathcal{C}_v$ at each prime $v\in (S\backslash \{p\})\cup X_m$ and $\mathcal{C}_p^{\lambda}$ at $p$. We are not done, since the twist is not known to satisfy the conditions $\mathcal{C}_v$ for $v\in X_{m+1}'\backslash X_m$. At each prime $v\in X_{m+1}'\backslash X_m$, choose a cohomology class $h_v\in H^1(\op{G}_{\Q_v},\g)$ such that $(\op{Id}+p^m h_v)\rho_{m+1| \op{G}_{\Q_v}}'$ satisfies $\mathcal{C}_v$. Note that $h_v$ is determined up to $\mathcal{N}_v$, i.e. if $g_v \in \mathcal{N}_v$ then $(\op{Id}+p^m (h_v+g_v))\rho_{m+1| \op{G}_{\Q_v}}'$ satisfies $\mathcal{C}_v$. According to Lemma $\ref{Nvnice}$, the space $\mathcal{N}_v$ consists of cohomology classes $h$ such that $h(\sigma_v)=0$. According to \cite[Corollary 11]{KLR} there is a finite set of nice primes $X_{m+1}$ containing $X_{m+1}'$ such that $\rho_{m+1}'$ satisfies $\mathcal{C}_v$ at each prime $v\in X_{m+1}\backslash X_{m+1}'$ and there is a cohomology class $g\in H^1(\op{G}_{\Q, S\cup X_{m+1}},\g)$ such that
        \begin{enumerate}
            \item $g_{\restriction \op{G}_{\Q_v}}=0$ for $v\in X_m$, 
            \item $g(\sigma_v)=h_v(\sigma_v)$ for $v\in X_{m+1}'\backslash X_m$,
            \item $g(\sigma_v)=0$ for $v\in X_{m+1}\backslash X_{m+1}'$.
        \end{enumerate}
        In other words, $g_{\restriction \op{G}_{\Q_v}}\in h_v+\mathcal{N}_v$ for $v\in X_{m+1}'\backslash X_m$ and $g_{\restriction \op{G}_{\Q_v}}\in \mathcal{N}_v$ for $v\in X_{m+1}\backslash X_{m+1}'$. Recall that $\rho_{m+1}'$ satisfies $\mathcal{C}_v$ for $v\in X_{m+1}\backslash X_{m+1}'$, thus, $(\op{Id}+p^{m} g)\rho_{m+1}'$ satisfies $\mathcal{C}_v$ at $v$ as well. The twist $(\op{Id}+p^{m} g)\rho_{m+1}'$ satisfies $(\star)$. This concludes the proof.
      \end{proof}
     
      \begin{Cor}\label{liftrholambda}
      There is a finite set of nice primes $Z$ containing $V$ such that there is a deformation \[\rho_{\lambda+2}:\op{G}_{\Q, S\cup Z}\rightarrow \op{GL}_2(\op{W}(\F)/p^{\lambda+2})\] of $\bar{\rho}$ satisfying $(\star)$.
      \end{Cor}
      \begin{proof}
           The Corollary is an immediate consequence of Lemma $\ref{closetotheend}$.
      \end{proof}
      \begin{proof}(of Theorem \ref{main})
      Corollary $\ref{liftrholambda}$ asserts that there is a finite set of nice primes $Z$ containing $V$ such that there is a deformation \[\rho_{\lambda+2}:\op{G}_{\Q, S\cup Z}\rightarrow \op{GL}_2(\op{W}(\F)/p^{\lambda+2})\] of $\bar{\rho}$ satisfying $(\star)$. There is a finite set of nice primes $X$ containing $Z$ such that the dual Selmer group $H^1_{\mathcal{N}^{\perp}}(\op{G}_{\Q,S\cup X}, \g^*)$ is zero. This is a standard argument, see the proofs of \cite[Lemma 1.2]{taylor} or \cite[Proposition 5.2, Lemma 5.3]{PatEx}. 
     \par We show that $\rho_{\lambda+2}$ may be inductively lifted to a characteristic zero representation
     \[\rho: \op{G}_{\Q,S\cup X}\rightarrow \op{GL}_2(\op{W}(\F))\] satisfying the condition $\mathcal{C}_p^{\lambda}$ at $p$. Moreover, $\rho$ shall satisfy $\mathcal{C}_v$ at the primes $v\in S\backslash \{p\}\cup X$ as well. Here, we recall that for $v\in S\backslash \{p\}$, the deformation functor $\mathcal{C}_v$ is chosen as in Proposition $\ref{prop23}$, and for $v\in X$, it is specified by Definition $\ref{nicedefs}$.
     \par We show by induction that $\rho_{\lambda+2}$ may be lifted to a characteristic zero representation $\rho$.
     Since $H^1_{\mathcal{N}^{\perp}}(\op{G}_{\Q,S\cup X}, \g^*)=0$, by Poitou-Tate, it follows that the restriction map
      \[H^1(\operatorname{G}_{\Q, S\cup X}, \g)\xrightarrow{\operatorname{res}_{S\cup X}} \bigoplus_{v\in S\cup X} \frac{H^1(\op{G}_{\Q_v}, \g)}{\mathcal{N}_v}\]is surjective. Let $m\geq \lambda+2$ and $\rho_m:\op{G}_{\Q,S\cup X}\rightarrow \op{GL}_2(\op{W}(\F)/p^m)$ be a lift of $\rho_{\lambda+2}$ which satisfies the specified local conditions $(\star)$ at the primes $S\cup X$ (see Definition $\ref{stardef}$).
      The inductive argument involves lifting $\rho_m$ one more step to $\rho_{m+1}$, so that the same conditions are satisfied. Recall from section $\ref{section32}$, that if the obstruction class $\mathcal{O}(\rho_m)$ is equal to zero, then $\rho_m$ must lift one more step. Since $\rho_m$ satisfies a liftable deformation condition at each prime $v\in S\cup X$, the obstruction class $\mathcal{O}(\rho_m)$ is trivial when restricted to a prime $v\in S\cup X$. Therefore, it belongs to $\Sha^2_{S\cup X} (\g)$. As explained in section $\ref{section32}$, since the dual Selmer group is zero, so is $\Sha^2_{S\cup X}(\g)$. Hence $\rho_m$ lifts to a representation \[\rho_{m+1}':\op{G}_{\Q,S\cup X}\rightarrow \op{GL}_2(\op{W}(\F)/p^{m+1}).\] It is shown that there is a suitable global cohomology class $z\in H^1(\operatorname{G}_{\Q, S\cup X}, \g)$, such that the twist \[\rho_{m+1}:=(\operatorname{Id}+p^{m}z)\rho_{m+1}'\] satisfies the specified local conditions $(\star)$ at the primes $S\cup X$. At each prime $v\in (S\backslash\{p\})\cup X$, there is a cohomology class $z_v\in H^1(\op{G}_{\Q_v}, \g)$ such that the twist \[(\operatorname{Id}+p^mz_v){\rho_{m+1}}_{\restriction \op{G}_{\Q_v}}\in\mathcal{C}_v(\op{W}(\F)/p^{m+1})\] and a class $z_p\in H^1(\op{G}_{\Q_p}, \g)$ such that \[(\operatorname{Id}+p^mz_p){\rho_{m+1}}_{\restriction \op{G}_{\Q_p}}\in \mathcal{C}_p^{\lambda}(\op{W}(\F)/p^{m+1}).\] Since $m\geq \lambda+2$, it follows from Proposition $\ref{localatp}$ that $\mathcal{N}_p$ stabilizes $\mathcal{C}_p^{\lambda}$. The surjectivity of $\operatorname{res}_{S\cup X}$ implies that the tuple
      \[(z_v)\in \bigoplus_{v\in S\cup X} H^1(\op{G}_{\Q_v}, \g)/\mathcal{N}_v\] arises from a global cohomology class $z$ which is unramified outside $S\cup X$. Therefore, $\rho_{m+1}$ satisfies the conditions $\mathcal{C}_v$ at each prime $v\in (S\backslash\{p\})\cup X$ and $\mathcal{C}_p^{\lambda}$ at $p$. In other words, it satisfies the condition $(\star)$. This completes the inductive lifting argument. Thus, there is a characteristic zero lift $\rho$ unramified away from $S\cup X$, satisfying $\mathcal{C}_p^{\lambda}$ at $p$.

    \par Recall that all deformations are assumed to have determinant $\chi$ and thus, it follows from the main theorem of \cite{KisinFM} that $\rho$ arises from a normalized eigencuspform $f=\sum_{n\geq 1} a_n q^n\in  S_2(\Gamma_1(N))$. Since $\rho$ is crystalline, the level $N$ is prime to $p$. Let $\mathfrak{p}$ be the prime above $p$ in $\Q(f)$ so that $\rho=\rho_{f,\mathfrak{p}}$. Since $\rho$ satisfies $\mathcal{C}_p^{\lambda}$ at $p$, it follows from Proposition $\ref{comparisonprop}$ that 
     \[v_p(\iota_{\mathfrak{p}}(a_p))=\lambda.\] The proof is complete.
      \end{proof}
      
\end{document}